\documentclass[11pt]{article}

\usepackage{geometry}
\usepackage[utf8]{inputenc}
\usepackage[english]{babel}
\usepackage[T1]{fontenc}
\usepackage{amsmath}
\usepackage{amsthm}
\usepackage{amsfonts}
\usepackage{amssymb}
\usepackage{lmodern}
\usepackage{verbatim}
\usepackage{ulem}

\usepackage{enumitem}

\usepackage[ruled]{algorithm}
\usepackage{algorithmic}

\usepackage{stmaryrd}

\usepackage{cancel}

\usepackage{hyperref}
\usepackage{url}

\usepackage{subfigure}
\usepackage{graphicx}

\usepackage{tikz}
\usepackage{tikz-3dplot}
\usetikzlibrary{3d}

\usetikzlibrary{positioning}
\usetikzlibrary{backgrounds}
\usetikzlibrary{patterns}
\usepackage{pgfplots}

\usepackage[font=small]{caption}

\usepackage{array}
\usepackage{graphbox}
\usepackage{xcolor,colortbl}

\usepackage{longtable}
\usepackage{booktabs}

\usepackage{wrapfig}
\usepackage{eurosym}
\usepackage{color}

\definecolor{Red}{rgb}{1,0,0}
\definecolor{Green}{rgb}{0,.6,0}
\definecolor{Blue}{rgb}{0,0,1}

\newcommand{\bl}{\color{black}} % black for final version

\usepackage{pifont}

\usepackage[title]{appendix}

\theoremstyle{definition}
\newtheorem{definition}{Definition}[section]

\theoremstyle{remark}
\newtheorem{remark}{Remark}[section]

\theoremstyle{plain}
\newtheorem{lemma}{Lemma}[section]

\theoremstyle{plain}
\newtheorem{theorem}{Theorem}[section]

\theoremstyle{plain}

\providecommand{\keywords}[1]
{
  \small
  \textbf{Key words.} #1
}

\usepackage{xspace}
\newcommand{\nomad}{{\sf NOMAD}\xspace}

\title{Efficient search strategies for constrained multiobjective blackbox optimization}

\author{
  \href{mailto:sebastien.le-digabel@polymtl.ca}{\bl S\'ebastien Le~Digabel}\thanks{
    GERAD and D\'epartement de math\'ematiques et de g\'enie industriel,
    \href{https://polymtl.ca}{\bl Polytechnique Montr\'eal},
    C.P. 6079, Succ. Centre-ville,
    Montr\'eal, Qu\'ebec, Canada H3C~3A7.
    \href{mailto:sebastien.le-digabel@polymtl.ca}{\bl \url{sebastien.le-digabel@polymtl.ca}},
    \href{https://www.gerad.ca/Sebastien.Le.Digabel}{\bl \url{www.gerad.ca/Sebastien.Le.Digabel}}.
  }
  \and
  \href{mailto:antoine.lesage-landry@polymtl.ca}{\bl Antoine Lesage-Landry}\thanks{
    Mila, GERAD, and D\'epartement de g\'enie \'electrique,
    \href{https://polymtl.ca}{\bl Polytechnique Montr\'eal},
    C.P. 6079, Succ. Centre-ville,
    Montr\'eal, Qu\'ebec, Canada H3C~3A7.
    \href{mailto:antoine.lesage-landry@polymtl.ca}{\bl \url{antoine.lesage-landry@polymtl.ca}},
    \href{https://www.gerad.ca/people/antoine-lesage-landry}{\bl \url{www.gerad.ca/antoine-lesage-landry}}.
  }
  \and
  \href{mailto:ludovic.salomon@polymtl.ca}{\bl Ludovic Salomon}\thanks{
    GERAD and D\'epartement de math\'ematiques et de g\'enie industriel,
    \href{https://polymtl.ca}{\bl Polytechnique Montr\'eal},
    C.P. 6079, Succ. Centre-ville,
    Montr\'eal, Qu\'ebec, Canada H3C~3A7.
    \href{mailto:ludovic.salomon@polymtl.ca}{\bl \url{ludovic.salomon@polymtl.ca}},
  }
  \and
  \href{mailto:christophe.tribes@polymtl.ca}{\bl Christophe Tribes}\thanks{
    GERAD and D\'epartement de math\'ematiques et de g\'enie industriel,
    \href{https://polymtl.ca}{\bl Polytechnique Montr\'eal},
    C.P. 6079, Succ. Centre-ville,
    Montr\'eal, Qu\'ebec, Canada H3C~3A7.
    \href{mailto:christophe.tribes@polymtl.ca}{\bl \url{christophe.tribes@polymtl.ca}},
    \href{https://www.gerad.ca/people/christophe-tribes}{\bl \url{www.gerad.ca/Christophe.Tribes}}.
  }
}

\date{}

\pgfplotsset{compat=1.18}

\begin{document}

\maketitle

\begin{abstract}
  Multiobjective blackbox optimization deals with problems where the objective and constraint functions are the outputs of a numerical simulation. In this context, no derivatives are available, nor can they be approximated by finite differences, which precludes the use of classical gradient-based techniques. The DMulti-MADS algorithm implements a state-of-the-art direct search procedure for multiobjective blackbox optimization based on the mesh adaptive direct search (MADS) algorithm. Since its conception, many search strategies have been proposed to improve the practical efficiency of the single-objective MADS algorithm. Inspired by this previous research, this work proposes the integration of two search heuristics into the DMulti-MADS algorithm. The first uses quadratic models, built from previously evaluated points, which act as surrogates for the true objectives and constraints, to suggest new promising candidates. The second exploits the sampling strategy of the Nelder-Mead algorithm to explore the decision space for new non-dominated points. Computational experiments on analytical problems and three engineering applications show that the use of such search steps considerably improves the performance of the DMulti-MADS algorithm.
\end{abstract}


% ----------------------------------------------------%
\section{Introduction}
% ----------------------------------------------------%

This work considers the following constrained multiobjective {\bl optimization} problem
\begin{equation} \label{ref:MOP} \tag{MOP}
  \min_{\mathbf{x} \in \Omega} f(\mathbf{x}) = \left(f_1(\mathbf{x}), f_2(\mathbf{x}), \ldots, f_m(\mathbf{x})\right)^\top.
\end{equation}
\(\Omega = \{\mathbf{x} \in \mathcal{X} : c_j(\mathbf{x}) \leq 0, \ j \in \mathcal{J}\} \subset \mathbb{R}^n\) represents the \textit{feasible decision set}, and \(\mathcal{X}\) is a subset of \(\mathbb{R}^n\). The sets \(\mathbb{R}^n\) and \(\mathbb{R}^m\) are, respectively, designed as the \textit{decision space} and the \textit{objective space}. Each objective component \(f_i : \mathbb{R}^n \rightarrow \mathbb{R} \cup \{+ \infty\}\), \(i = 1,2,\ldots, m\), and constraints \(c_j : \mathbb{R}^n \rightarrow \mathbb{R} \cup \{+ \infty\}\), \(j \in \mathcal{J}\) are blackbox functions, for which derivatives are not available nor cannot be approximated numerically. Blackbox optimization (BBO) algorithms aim to solve such problems.

In the absence of analytical expressions, one cannot exploit certain properties of the problem (e.g., continuity, convexity, or differentiability), justifying the use of derivative-free optimization (DFO)~\cite{AuHa2017, CoScVibook, LaMeWi2019}, a more general case of BBO.
Many engineering applications that face several conflicting, expensive, and/or unreliable objectives fit into this framework, e.g.,~\cite{Alexandropoulos2019, DiPiKhuSaBe2009, Fang2005, Sharma2013}. For such problems, an optimal solution for all objectives does not always exist. The goal is then to provide the practitioners with the best set of trade-off solutions from which they can extract a particular solution according to their preferences~\cite{Branke2008Multiobjective, Collette2005, Miettinen_99_a}.

Prior work has already proposed several deterministic algorithms to address this class of problems. A first approach is to reformulate the original multiobjective optimization problem into a series of single-objective subproblems, each of which is solved by a dedicated single-objective blackbox solver. Direct search algorithms such as BiMADS~\cite{AuSaZg2008a} and MultiMADS~\cite{AuSaZg2010a} for biobjective and multiobjective optimization, respectively, rely on the Mesh Adaptive Direct Search (MADS)~\cite{AuDe2006, AuDe09a} algorithm for single-objective constrained optimization. The BiObjective Trust-Region algorithm~\cite{Ryubiobjdftr} adopts a derivative-free trust-region method. Methods with a posteriori articulation of preferences work directly with the original formulation, in contrast to the scalarization-based approaches mentioned above. The Direct MultiSearch (DMS) framework~\cite{CuMaVaVi2010} and its variants~\cite{BiLedSa2020, G-2022-10, CuMa2018, DedDesNa2021} generalize single-objective direct search algorithms to multiobjective optimization. The Derivative-Free MultiObjective algorithm~\cite{LiLuRi2016} is a linesearch-based approach, while the MultiObjective Implicit Filtering method~\cite{moif2018} can be viewed as a multiobjective extension of the implicit filtering algorithm~\cite{Kelley2011}.

Direct search methods are DFO iterative methods. At each iteration, they examine a finite set of trial candidates to try to improve a current (set of) incumbent solution(s). If an improvement is found, the current (set of) incumbent solution(s) is updated. Otherwise, the associated step size parameter of the current incumbent is decreased, and a new iteration is started. Most efficient direct search algorithms for single-objective optimization such as MADS~\cite{AuDe2006} or Generated Search Set~\cite{KoLeTo03a} alternate between search and poll steps.
Search strategies, although not mandatory for convergence, significantly enhance the practical performance of direct search algorithms.
Thus, many search strategies have been proposed for single-objective optimization: speculative search~\cite{AuDe2006}, quadratic models~\cite{CoLed2011,CuRoVi10,CuVi07}, Variable Neighbourhood Search~\cite{AuBeLe08}, general surrogate models~\cite{AuKoLedTa2016,TaAuKoLed2016}, or a Nelder-Mead (NM) based search strategy~\cite{AuTr2018}.

In multiobjective optimization, the integration of search strategies can take two forms. Scalarization-based approaches such as BiMADS~\cite{AuSaZg2008a} or MultiMADS~\cite{AuSaZg2010a} incorporate the direct search single-objective solver MADS, which in turns implements a certain number of single-objective search heuristics. However, computational experiments~\cite{BiLedSa2020, G-2022-10, CuMaVaVi2010} have shown that even with these search strategies activated, the algorithms have similar performance with state-of-the-art direct search methods with a posteriori articulation of preferences using only polling.
References~\cite{BraCu2020, CuMa2018} propose ``native'' search strategies specifically designed for multiobjective direct search algorithms. 

This work proposes two new native search strategies for the multiobjective direct search algorithm DMulti-MADS~\cite{BiLedSa2020, G-2022-10}. The first one is based on quadratic models in the lineage of~\cite{BraCu2020}. The second search method is an adaptation of a NM strategy proposed in~\cite{AuTr2018} for multiobjective optimization.

Note that the use of surrogate models in the context of multiobjective blackbox optimization is not new. The following works~\cite{Akhtar2016, Mueller2017, Regis2016a} combine radial basis functions with evolutionary strategies. Some research has also explored Gaussian processes within a Bayesian multiobjective framework, e.g.,~\cite{Bradford2018, FeBeVa2017}. For such methods, convergence analysis is not well established, or at most in probability (see, for example,~\cite{Regis2016b}). On the contrary, the following works~\cite{Ryubiobjdftr, MoHaGa2025, PrThEiBoSc2020, Thomann2019, ThEich2019} describe multiobjective algorithms using quadratic models within a trust-region framework with convergence guarantees. However, these algorithms are limited to unconstrained or bound-constrained multiobjective optimization.

Researchers have also proposed NM-based algorithms for multiobjective optimization. The authors in~\cite{ZaCol2015} combine a NM strategy with a Chebyshev scalarization. In~\cite{MeDas2020}, the authors adopt an evolutionary strategy based on the NM simplex ordering. However, all these methods do not provide convergence guarantees and are not deterministic, potentially making them less robust that existing convergent methods. 

Based on these considerations, this work:
\begin{itemize}
    \item proposes two search strategies for constrained multiobjective blackbox optimization, motivated by their practical performance already observed in a single-objective context;
    \item investigates quadratic-based formulations other than the ones proposed by~\cite{BraCu2020}, whose main drawback is the number of subproblems to solve, which increases exponentially with the number of objectives. However, this issue remains limited for a low number of objectives, i.e., \(m \leq 4\). The authors have also investigated the use of parallelism to speed up this search~\cite{Tavares2023}. The formulations used in this work exploit the nature of the incumbent iterate, resulting in less problems to solve (i.e., at most linear in the number of objectives), and keeping a performance similar to the search strategy proposed in~\cite{BraCu2020}.
\end{itemize}

This document is organized as follows. Sections~\ref{sect:Pareto_dominance} and~\ref{sect:necessary_optimality_conditions}  summarize the main notations and multiobjective concepts, and optimality conditions. Section~\ref{sect:Presentation_DMulti_MADS} provides a general presentation of the DMulti-MADS algorithm. New single-objective formulations are given in Section~\ref{sect:Single-objective reformulations}. The new search strategies are then proposed in Section~\ref{sect:Search steps}. This work concludes with numerical experiments in Section~\ref{sect:Numerical_experiments}, which illustrate the potential of such new strategies in analytical benchmarks and three applications. It ends with a general conclusion.

% ----------------------------------------------------%
\section{Notations and Pareto dominance}
\label{sect:Pareto_dominance}
% ----------------------------------------------------%

This work adopts the following conventions. Given two vectors \(\mathbf{y}^1, \ \mathbf{y}^2 \in \mathbb{R}^m\),
\begin{itemize}
    \item \(\mathbf{y}^1 = \mathbf{y}^2 \Longleftrightarrow y^1_i = y^2_i\) for \(i = 1, 2, \ldots, m\);
    \item \(\mathbf{y}^1 < \mathbf{y}^2 \Longleftrightarrow y^1_i < y^2_i\) for \(i = 1, 2, \ldots, m\);
    \item \(\mathbf{y}^1 \leq \mathbf{y}^2 \Longleftrightarrow y^1_i \leq y^2_i\) for \(i = 1, 2, \ldots, m\) and there exists at least an index \(i_0 \in \{1, 2, \ldots, m\}\) such that \(y^1_{i_0} < y^2_{i_0}\).
\end{itemize}

In multiobjective optimization, as the objective function possesses several components, one needs the concept of \textit{Pareto dominance}~\cite{Miettinen_99_a} to compare a pair of points.

\begin{definition}
    Let two feasible decision vectors \(\mathbf{x}^1\) and \(\mathbf{x}^2\) in \(\Omega\).
    \begin{itemize}
        \item \(\mathbf{x}^1\) \textit{dominates} \(\mathbf{x}^2\) (\(\mathbf{x}^1 \prec \mathbf{x}^2\)) if and only if \(f(\mathbf{x}^1) \leq f(\mathbf{x}^2)\).
        \item \(\mathbf{x}^1\) \textit{strictly dominates} \(\mathbf{x}^2\) (\(\mathbf{x}^1 \prec \prec \mathbf{x}^2\)) if and only if \(f(\mathbf{x}^1) < f(\mathbf{x}^2)\).
        \item \(\mathbf{x}^1\) and \(\mathbf{x}^2\) are \textit{incomparable} (\(\mathbf{x}^1 \sim \mathbf{x}^2\)) if and only if \(\mathbf{x}^1 \nprec \mathbf{x}^2\) and \(\mathbf{x}^2 \nprec \mathbf{x}^1\), i.e. \(f(\mathbf{x}^1) \nleq f(\mathbf{x}^2)\) and \(f(\mathbf{x}^2) \nleq f(\mathbf{x}^1)\).
    \end{itemize}
\end{definition}

This definition can be extended to objective vectors, i.e., points in the objective space \(\mathbb{R}^m\). It is now possible to define what an optimal solution for (\ref{ref:MOP}) is.

\begin{definition}
  A feasible decision vector \(\mathbf{x}^\star \in \Omega\) is said to be (globally) \textit{Pareto optimal} if there is no other decision vector \(\mathbf{x}\in \Omega\) such that \(\mathbf{x} \prec \mathbf{x}^\star\).
\end{definition}

The set of all Pareto optimal solutions in \(\Omega\) is called the \textit{Pareto set} denoted by \(\mathcal{X}_\text{P}\) and its image by the objective function is designated as the \textit{Pareto front} denoted by \(\mathcal{Y}_\text{P} \subseteq \mathbb{R}^m\).

\begin{definition}
  A feasible decision vector \(\mathbf{x}^\star \in \Omega\) is said to be \textit{locally Pareto optimal} if it exists a neighbourhood \(\mathcal{N}(\mathbf{x}^\star)\) of \(\mathbf{x}^\star\) such that there is no other decision vector \(\mathbf{x} \in \mathcal{N}(\mathbf{x}^\star) \cap \Omega\) which satisfies \(\mathbf{x} \prec \mathbf{x}^\star\).
\end{definition}

Any set of locally Pareto optimal solutions is called a \textit{local Pareto set}. The Pareto set may contain an infinite number of solutions~\cite{Collette2011Multiobjective}, making its enumeration impractical. Furthermore, obtaining even locally optimal solutions in a DFO context is a challenging task. Algorithms try to find a representative set of non-dominated points, i.e., a \textit{Pareto set approximation}~\cite{ZiKnTh2008}, whose mapping by the objective function \(f\) is designed as a \textit{Pareto front approximation}. In the best case, all elements of a Pareto set approximation should be (locally) Pareto optimal, but this condition is not always satisfied.

The two following objective vectors provide information on the range of the Pareto front. The \textit{ideal objective vector}~\cite{Miettinen_99_a} \(\mathbf{y}^\text{I} \in \mathbb{R}^m\) (if it exists) bounds the Pareto front from below and is defined as
\[\mathbf{y}^\text{I} = \left(\min_{\mathbf{x} \in \Omega} f_1(\mathbf{x}), \min_{\mathbf{x} \in \Omega} f_2(\mathbf{x}), \ldots, \min_{\mathbf{x} \in \Omega} f_m(\mathbf{x}) \right)^\top.\]
The ideal objective vector is related to the \textit{extreme points of the Pareto set}, i.e., the elements of the Pareto set that are the solutions of each single-objective problem \(\min_{\mathbf{x} \in \Omega} f_i(\mathbf{x})\) for \(i = 1,2, \ldots, m\).
The \textit{nadir objective vector}~\cite{Miettinen_99_a} \(\mathbf{y}^{\text{N}} \in \mathbb{R}^m\) (if it exists) provides an upper bound on the Pareto front. It is defined as
\[\mathbf{y}^{\text{N}} = \left(\max_{\mathbf{x} \in \mathcal{X}_P} f_1(\mathbf{x}), \max_{\mathbf{x} \in \mathcal{X}_\text{P}} f_2(\mathbf{x}), \ldots, \max_{\mathbf{x} \in \mathcal{X}_\text{P}} f_m(\mathbf{x})\right)^\top.\]

% ----------------------------------------------------%
\section{Necessary optimality conditions}
\label{sect:necessary_optimality_conditions}
% ----------------------------------------------------%
This section contains a summary of the necessary optimality conditions for multiobjective optimization used in the rest of this work. Clarke nonsmooth analysis~\cite{Clar83a} provides a rigorous framework for studying derivative-free multiobjective methods in the constrained case.

The Clarke tangent cone is a generalization of the tangent cone used in classical differentiable nonlinear optimization. Roughly speaking, it defines a set of directions in the decision space that an algorithm can use to move from one feasible solution to another. The definition and notations are taken from~\cite{AuDe2006}.

\begin{definition}[Clarke tangent vector]
    A vector \(\mathbf{d} \in \mathbb{R}^n\) is said to be a Clarke tangent vector to the set \(\Omega \subseteq \mathbb{R}^n\) at the point \(\mathbf{x}\) in the closure of \(\Omega\) if for every sequence \(\{\mathbf{y}^k\}\) of elements of \(\Omega\) that converge to \(\mathbf{x}\) and every sequence of positive real numbers \(\{t^k\}\) converging to zero, there exists a sequence of vectors \(\{\mathbf{w}^k\}\) converging to \(\mathbf{d}\) such that \(\mathbf{x}^k + t^k \mathbf{w}^k \in \Omega\).
\end{definition}

The set of all Clarke tangent vectors to \(\Omega\) at \(\mathbf{x}\) is called the Clarke tangent cone to \(\Omega\) at \(\mathbf{x}\), denoted as \(\mathcal{T}^{\text{Cl}}_\Omega(\mathbf{x})\). The hypertangent cone is the interior of the Clarke tangent cone, assuming that the latter exists. It is often used in the convergence analysis of direct search methods.

\begin{definition}[Hypertangent vector]
    A vector \(\mathbf{d} \in \mathbb{R}^n\) is said to be a hypertangent vector to the set \(\Omega \subseteq \mathbb{R}^n\) at \(\mathbf{x}\) in \(\Omega\) if there exists a scalar \(\epsilon > 0\) such that
    \[\mathbf{y} + t\mathbf{w} \in \Omega, \forall \mathbf{y} \in \Omega \cap \mathcal{B}(\mathbf{x}; \epsilon), \mathbf{w} \in \mathcal{B}(\mathbf{x}; \epsilon) \text{ and } 0 < t < \epsilon.\]
\end{definition}

The set of all hypertangent vectors at \(\mathbf{x} \in \Omega\) is called the hypertangent cone at \(\mathbf{x}\) in \(\Omega\), denoted as \(\mathcal{T}^\text{H}_\Omega(\mathbf{x})\).

Assuming that \(f\) is Lipschitz continuous near \(\mathbf{x} \in \Omega\), i.e., each component \(f_i\) for \(i = 1, 2, \ldots, m\) of the objective function is Lipschitz continuous near \(\mathbf{x}\), one can define the Clarke-Jahn generalized derivative~\cite{Jahn2007} of each objective component in the direction \(\mathbf{d}\) belonging to the hypertangent cone at \(\mathbf{x}\) in \(\Omega\):
\[f^\text{o}_i(\mathbf{x}; \mathbf{d}) = \underset{\begin{array}{c}\mathbf{y} \rightarrow \mathbf{x}, \ \mathbf{y} \in \Omega \\ t \downarrow 0, \ \mathbf{y} + t \mathbf{d} \in \Omega\end{array}}{\lim \sup} \dfrac{f_i(\mathbf{y} + t\mathbf{d}) - f(\mathbf{y})}{t}, \text{ for all } i = 1, 2, \ldots, m.\]
The Clarke-Jahn generalized derivative of objective \(f_i\) in the direction \(\mathbf{v}\) belonging to the Clarke tangent cone at \(\mathbf{x}\) in \(\Omega\) can be found by going to the limit, i.e., \(f^\text{o}_i(\mathbf{x}; \mathbf{v}) = \underset{\begin{array}{c}\mathbf{d} \in \mathcal{T}^\text{H}_{\Omega}(\mathbf{x}) \\ \mathbf{d} \rightarrow \mathbf{v}\end{array}}{\lim} f^\text{o}_i(\mathbf{x};\mathbf{d})\)~\cite{AuDe2006}.

It is now possible to give the following main necessary condition for nonsmooth multiobjective optimization.

\begin{theorem}
    Let \(f\) be Lipschitz continuous near \(\hat{\mathbf{x}} \in \Omega\). If \(\hat{\mathbf{x}}\) is locally Pareto optimal, then for any direction \(\mathbf{d} \in \mathcal{T}^\text{Cl}_{\Omega}(\hat{\mathbf{x}})\), there exists at least one index \(i \in \{1, 2, \ldots, m\}\) such that:
    \[f^\text{o}_i(\hat{\mathbf{x}}; \mathbf{d}) \geq 0.\]
\end{theorem}

The above theorem means that there does not exist any descent direction that is a descent direction for all objectives at a locally Pareto optimal solution.

% ----------------------------------------------------%
\section{DMulti-MADS}
\label{sect:Presentation_DMulti_MADS}
% ----------------------------------------------------%

This work considers the DMulti-MADS algorithm~\cite{BiLedSa2020} with the progressive barrier technique~\cite{G-2022-10} to handle inequality constraints. DMulti-MADS is an iterative direct search method that extends the MADS algorithm~\cite{AuDe2006, AuDe09a} to multiobjective constrained optimization. Inspired by BiMADS~\cite{AuSaZg2008a}, it can be seen as a specific instance of the DMS framework~\cite{CuMaVaVi2010}. Good numerical results on benchmark functions~\cite{BiLedSa2020} and real applications justify its use in this research.

Each iteration of DMulti-MADS consists of two main steps: the \textit{search} and the \textit{poll}. The search is an optional step that allows the use of various strategies, such as surrogates~\cite{BraCu2020}, to explore the decision space by evaluating a finite number of candidates. The poll performs a local exploration of the decision space in a region delimited by a frame size parameter \(\Delta^k > 0\) centered around a current incumbent. The poll follows stricter rules, which, when properly configured, ensures convergence~\cite{BiLedSa2020}.

All candidates generated during an iteration \(k\) must belong to a discrete set called the \textit{mesh} defined by
\[M^k = \bigcup_{\mathbf{x} \in V^k} \{\mathbf{x} + \delta^k \mathbf{D} \mathbf{z}: \mathbf{z} \in \mathbb{N}^{n_\text{D}}\} \subset \mathbb{R}^n,\]
where \(\delta^k > 0\) is the mesh size parameter, \(\mathbf{D} = \mathbf{G} \mathbf{Z} \in \mathbb{R}^{n \times n_\text{D}}\) is a positive spanning set matrix, i.e., a matrix whose columns form a positive spanning set for \(\mathbb{R}^n\) (see~\cite[Chapter 6]{AuHa2017} or~\cite[Chapter 2]{CoScVibook}) for some non-singular matrix \(\mathbf{G} \in \mathbb{R}^{n \times n}\) and integer matrix \(\mathbf{Z} \in \mathbb{Z}^{n \times n_\text{D}}\). In practice, \(\mathbf{G} = I_n\) and \(\mathbf{Z} = [-I_n \ I_n] = \mathbf{D}\), where \(I_n\) is the identity matrix of size \(n \times n\). The \textit{cache} \(V^k \subset \mathcal{X}\) is the set of trial points which have been evaluated by the algorithm by the start of iteration \(k\).

The poll set at iteration \(k\), denoted as \(P^k\), is a subset of \(M^k\). Its construction must satisfy specific requirements. It involves a current iterate incumbent \((\mathbf{x}^k, \Delta^k) \in V^k \times \mathbb{R}_+^*\). The point \(\mathbf{x}^k\) is denoted as the current frame incumbent center at iteration \(k\). The parameter \(\Delta^k\) is called the \textit{frame size parameter} at iteration \(k\) and must satisfy \(0 < \delta^k \leq \Delta^k\) for all \(k\) and \(\lim_{k \in K} \delta^k = 0\) if and only if \(\lim_{k \in K} \Delta^k = 0\) for any subset of indices \(K\). In practice, \(\delta^k = \min\{\Delta^k, \left(\Delta^k\right)^2\}\) satisfies these requirements. Formally, 
\[P^k = \{\mathbf{x}^k + \delta^k \mathbf{d}: \mathbf{d} \in \mathbb{D}^k_\Delta\} \subset \{\mathbf{x} \in M^k :\| \mathbf{x} - \mathbf{x}^k\|_{\infty} \leq \Delta^k b\},\]
where \(b  = \max \{\|\mathbf{d}'\|_{\infty}: \mathbf{d}' \in \mathbf{D} \}\), \(\mathbb{D}^k_{\Delta}\) is a positive spanning set of directions. Interested readers can refer to~\cite{AbAuDeLe09, AuDe2006} for procedures to build such positive spanning sets \(\mathbb{D}^k_\Delta\).

To handle inequality constraints, DMulti-MADS with the progressive barrier uses the constraint violation function~\cite{FlLe02a, FlLeTo02a}:
\[h(\mathbf{x}) =
\begin{cases}
\displaystyle\sum_{j \in \mathcal{J}} (\max \left\{0, c_j(\mathbf{x})\right\})^2 & \text{ if } \mathbf{x} \in \mathcal{X}, \\
+ \infty & \text{ otherwise.}
\end{cases}\]
The constraint violation function quantifies the violation of the constraints at a given point. It is non-negative and satisfies \(h(\mathbf{x}) = 0 \) if and only if \(\mathbf{x} \in \Omega\).

Using this constraint violation function, DMulti-MADS introduces an extension of the dominance relation for constrained multiobjective optimization~\cite{G-2022-10}.

\begin{definition}[Dominance for constrained optimization]
    Let two decision vectors \(\mathbf{x}^1\) and \(\mathbf{x}^2\) be in \(\mathcal{X}\). Then \(\mathbf{x}^1 \in \mathcal{X}\) is said to dominate \(\mathbf{x}^2 \in \mathcal{X}\) if
    \begin{itemize}
        \item both points are feasible and \(\mathbf{x}^1 \in \Omega\) dominates \(\mathbf{x}^2 \in \Omega\), denoted as \(\mathbf{x}^1 \prec_f \mathbf{x}^2\);
        \item both points are infeasible and \(f_i(\mathbf{x}^1) \leq f_i(\mathbf{x}^2)\) for \(i = 1, 2, \ldots, m\) and \(h(\mathbf{x}^1) \leq h(\mathbf{x}^2)\) with at least one inequality strictly satisfied, denoted as \(\mathbf{x}^1 \prec_h \mathbf{x}^2\).
    \end{itemize}
\end{definition}

Note that a pair of feasible and infeasible points are never compared.
 
Each iteration of DMulti-MADS is organized around at least one of the following set of current incumbent solutions:
\begin{itemize}
    \item the set of feasible incumbent solutions
    \[F^k = \arg \min_{\mathbf{x} \in V^k} \{ f(\mathbf{x}) : \mathbf{x} \in \Omega\},\]
    \item the set of infeasible incumbent solutions
    \[I^k = \arg \min_{\mathbf{x} \in U^k} \{f(\mathbf{x}) : 0 < h(\mathbf{x}) \leq h^k_{\max}\},\]
    with \(U^k\) the set of infeasible nondominated points
    \[U^k = \{\mathbf{x} \in V^k \setminus \Omega: \text{ there is no other } \mathbf{y} \in V^k \setminus \Omega \text { such that } \mathbf{y} \prec_h \mathbf{x}\}\]
    and \(h^k_{\max} > 0\) the barrier threshold at iteration \(k\).
\end{itemize}

Using one of these two sets of current incumbents, DMulti-MADS constructs a so-called \textit{iterate list}~\cite{CuMaVaVi2010} of non-dominated points
\[L^k = \{(\mathbf{x}^l, \Delta^l) : \mathbf{x}^l \in X^k, l = 1, 2, \ldots, l_k\},\]
where \(l_k = |X^k|\) and \(X^k \subseteq \begin{cases} I^k & \text {if } F^k = \emptyset \\ F^k & \text{otherwise}\end{cases}\), with \(X_k \neq \emptyset\). Each element of the list possesses its own frame and mesh size parameters.

At iteration \(k\), DMulti-MADS chooses the current incumbent iterate \((\mathbf{x}^k, \Delta^k)\) among the elements of~\(L^k\). This last one must satisfy at least the following condition:
\[(\mathbf{x}^k, \Delta^k) \in \{(\mathbf{x}, \Delta) \in L^k : \tau^{w^+} \Delta^k_{\max} \leq \Delta\},\]
where \(\tau \in \mathbb{Q} \cap (0, 1)\) is the \textit{mesh size adjustment parameter}~\cite{AuHa2017}, \(w^+ \in \mathbb{N}\) a fixed integer, and \(\Delta^k_{\max}\) the maximum frame size parameter at iteration \(k\) defined as \(\Delta^k_{\max} = \displaystyle\max_{l = 1, 2, \ldots, |L^k|} \Delta^l\).

When \(F^k \neq \emptyset\) and \(I^k \neq \emptyset\), the poll may be executed around two frame centers belonging to these two sets with the frame size parameter \(\Delta^k\) of the current iteration \(k\).

All new candidates generated during the poll and the search are affected a frame size parameter larger than the frame size parameter of the current incumbent iterate. The frame size parameter of the current incumbent is reduced at the end of the iteration if no other candidate is found which dominates at least one of the current frame centers. In all cases, the iterate list is filtered at the end of the iteration to remove new dominated points for the next iteration and the barrier threshold is reduced or kept constant, i.e., \(h^{k+1}_{\max} \leq h^k_{\max}\).

Algorithm~\ref{alg:summary_DMulti-MADS_algorithm} provides a high-level description of DMulti-MADS for constrained optimization. Interested readers are referred to~\cite{BiLedSa2020, G-2022-10} for detailed information.

\begin{figure}[!th]
  \begin{algorithm}[H]\small
    \caption{A high-level description of the DMulti-MADS algorithm for constrained optimization, inspired by~\cite{BiLedSa2020, G-2022-10}.}
   %\caption{The DMulti-MADS algorithm for constrained optimization}
   \begin{algorithmic}[1]
    \STATE \textbf{Initialization}: Given a finite set of points \(V^0 \subset \mathcal{X}\), choose \(\Delta^0 > 0\), \(\mathbf{D} = \mathbf{G} \mathbf{Z}\) a positive spanning set matrix, \(\tau \in (0, 1) \cap \mathbb{Q}\) the frame size adjustment parameter, \(w^+ \in \mathbb{N}\) a fixed integer parameter, and \(h^0_{\max} = + \infty\) the initial barrier threshold. Initialize the iterate list \(L^0 = \{(\mathbf{x}^l, \Delta^0), l = 1, 2, \ldots, |L^k|\}\) for some \(\mathbf{x}^l \in V^0 \cap F^0\) or (exclusively) \(\mathbf{x}^l \in V^0 \cap I^0\).
    \FOR{\(k = 0, 1, 2, \ldots\)}
    \STATE \textbf{Selection of the current frame centers}. Select the current incumbent iterate \((\mathbf{x}^k, \Delta^k)\) of the iterate list \(L^k\), satisfying:
    \[(\mathbf{x}^k, \Delta^k) \in \{(\mathbf{x}, \Delta) \in L^k : \tau^{w^+} \Delta^k_{\max} \leq \Delta\},\]
    and its associated second frame center it it exists.
    
    Set \(\delta^k = \min \left\{\Delta^k, \left(\Delta^k\right)^2\right\}\). Initialize \(L^\text{add} := \emptyset\).
    \STATE \textbf{Search} (optional): Evaluate \(f\) and \(h\) at a finite set of points \(S^k \subset \mathcal{X}\) on the mesh \(M^k\). Set \(L^\text{add} := \{(\mathbf{x}, \Delta^k): \mathbf{x} \in S^k\}\).

    If a success criterion is satisfied, the search may terminate. In this case, skip the poll and go to the parameter update step.
    \STATE \textbf{Poll}: Select a positive spanning set \(\mathbb{D}^k_{\Delta}\). Evaluate \(f\) and \(h\) on the poll set \(P^k \subset M^k\) using at least the current frame center. Set \(L^\text{add} := L^\text{add} \cup \left\{(\mathbf{x}, \Delta^k) : \mathbf{x} \in P^k \right\}\).
    If a success criterion is satisfied, the poll may terminate opportunistically.
    \STATE \textbf{Parameter update}: Define \(V^{k+1}\) as the union of \(V^k\) and all new candidates evaluated in \(\mathcal{X}\) during the search and the poll. Set \(L^{k+1} := L^k\). Update \(h^{k+1}_{\max} \leq h^{k}_{\max}\). Update the iterate list \(L^{k+1}\) by adding new non-dominated points from \(L^\text{add}\) with their updated associated frame center \(\Delta \in \{\Delta^k, \tau^{-1} \Delta^k\}\). Remove new dominated points from \(L^{k+1}\).

    If the iteration is unsuccessful, replace the iterate element \((\mathbf{x}^k, \Delta^k)\) by \((\mathbf{x}^k, \Delta^{k+1})\) with \(\Delta^{k+1} := \tau \Delta^k\).
    \ENDFOR 
   \end{algorithmic}
   \label{alg:summary_DMulti-MADS_algorithm}
  \end{algorithm}
\end{figure}

If all trial points generated during the optimization belong to a bounded set and the set of refining directions~\cite{AuDe2006} is ``sufficiently rich'', the main convergence result of DMulti-MADS~\cite{BiLedSa2020} states that it will generate at least one subsequence converging to a point \(\mathbf{x}^{*}\) such that the Clarke generalized derivative~\cite{Clar83a} \(f_{i_0}^\text{o}(\mathbf{x}^{*}; \mathbf{d})\) will be nonnegative for at least one objective index \(i_0 \in \{1, 2, \ldots, m\}\) for every hypertangent direction \(\mathbf{d}\)~\cite{Jahn2007} to the domain \(\Omega\) at \(\mathbf{x}^*\), assuming the existence of feasible iterates.

When there is no feasible iterate, one can show similarly~\cite{G-2022-10} that DMulti-MADS will generate an accumulation infeasible point \(\mathbf{x}^{*}\) such that its Clarke generalized derivative \(h^\text{o}(\mathbf{x}^{*}; \mathbf{d})\) will be nonnegative for every hypertangent direction \(\mathbf{d}\) to the domain \(\mathcal{X}\) at \(\mathbf{x}^*\).

All these results hold independently of the search used to build \(S^k\), as long as this set is finite and search candidates belong to the mesh \(M^k\). The new search strategies described in the rest of this work meet these two requirements.

% ----------------------------------------------------%
\section{Single-objective formulations}
\label{sect:Single-objective reformulations}
% ----------------------------------------------------%

The search strategies proposed in this work are based on the resolution of single-objective subproblems derived from the original multiobjective optimization problem (\ref{ref:MOP}), in the continuation of~\cite{AuSaZg2008a, AuSaZg2010a, BraCu2020}.

\begin{definition}[adapted from~\cite{AuSaZg2008a, AuSaZg2010a}] \label{def:single_objective_formulation} Consider the single-objective optimization problem:
\[R_r: \min_{\mathbf{x} \in \Omega} \psi_{Y}(\mathbf{x}) \text{ with } \psi_{Y}(\mathbf{x}) = \phi_{Y} \left(f(\mathbf{x})\right),\]
where \(\phi_{Y} : \mathbb{R}^m \rightarrow \mathbb{R}\) is parameterized with respect to some finite reference set \(Y \subset \mathbb{R}^m \neq \emptyset\) satisfying
\[\text{for all } (\mathbf{r}^1, \mathbf{r}^2) \in Y, \mathbf{r}^1 \nleq \mathbf{r}^2 \text{ and } \mathbf{r}^2 \nleq \mathbf{r}^1.\]
Then \(R_r\) is called a \textit{single-objective formulation at \(Y\) of (\ref{ref:MOP}}) if the following conditions hold:
\begin{itemize}
    \item if \(f\) is Lipschitz continuous near some \(\tilde{\mathbf{x}} \in \Omega\), then \(\psi_Y\) is also Lipschitz continuous near \(\tilde{\mathbf{x}} \in \Omega\).
    \item if \(f\) is Lipschitz continuous near some \(\tilde{\mathbf{x}} \in \Omega\) with \(f_i(\tilde{\mathbf{x}}) < r_i\) component-wise for at least one element \(\mathbf{r} \in Y\) and if \(\mathbf{d} \in \mathcal{T}_{\Omega}^\text{Cl}(\tilde{\mathbf{x}})\) such that \(f_i^\text{o}(\tilde{\mathbf{x}}; \mathbf{d}) < 0\) for \(i = 1, 2, \ldots, m\), then \(\psi_Y^\text{o}(\tilde{\mathbf{x}}; \mathbf{d}) < 0\).
\end{itemize}
\end{definition}

Such formulation is useful as it preserves local Lipschitz continuity (first condition of Definition~\ref{def:single_objective_formulation}) and ensures that a descent direction for the objective function is also a descent direction for the single-objective formulation function (second condition of Definition~\ref{def:single_objective_formulation}).

This definition differs from the one proposed in~\cite{AuSaZg2008a, AuSaZg2010a}, as it is no more restricted to the use of only one reference vector. Note that the \textit{single-objective distance formulation}, defined in~\cite{AuSaZg2010a} by
\[
    \begin{array}{lll}
    \bar{R}_\mathbf{r}: \displaystyle\min_{\mathbf{x} \in \Omega} \bar{\psi}_{\mathbf{r}}(\mathbf{x}) & = & \bar{\phi}_{\mathbf{r}}(f(\mathbf{x})) \\
    & = & \begin{cases}
          -\text{dist}^{2}(\partial \mathcal{D}(\mathbf{r}), f(\mathbf{x})) & \text{if } f(\mathbf{x}) \in \mathcal{D}(\mathbf{r}),\\
          ~~ \text{dist}^{2}(\partial \mathcal{D}(\mathbf{r}), f(\mathbf{x})) & \text{otherwise},
        \end{cases}
    \end{array}
\]
where \(\text{dist}^{2}(\partial \mathcal{D}(\mathbf{r}), f(\mathbf{x}))\) is the distance from \(f(\mathbf{x})\) to the boundary \(\partial \mathcal{D}(\mathbf{r})\) of the dominance zone \(\mathcal{D}(\mathbf{r}) = \left\{\mathbf{y} \in \mathbb{R}^m: \mathbf{y} \leq \mathbf{r}\right\}\) relative to some reference objective vector \(\mathbf{r} \in \mathbb{R}^m\), satisfies the two conditions of Definition~\ref{def:single_objective_formulation} by considering the singleton \(Y = \{\mathbf{r}\}\).

By allowing the single-objective formulation to be parameterized by a reference set, an algorithm can potentially explore a larger zone of the objective space where only the zone dominated by a reference objective vector would have been prioritized. This work also uses the \textit{single-objective dominance move formulation}, inspired by~\cite{Li2017} and defined as
\[
    \begin{array}{lll}
    R^\text{Do}_{Y}: \displaystyle\min_{\mathbf{x} \in \Omega} \psi^\text{Do}_{Y}(\mathbf{x}) & = & \phi^\text{Do}_{Y}\left(f(\mathbf{x})\right) \\
    & = & 
    \begin{cases}
          - \displaystyle\min_{\mathbf{r} \in Y} \displaystyle\sum_{i = 1}^{m} \max\left\{0, r_{i} - f_{i}(\mathbf{x})\right\} & \text{if } f(\mathbf{x}) \text{ is not dominated by any element of } Y,  \\
          ~~~\displaystyle\min_{\mathbf{r} \in Y} \displaystyle\sum_{i = 1}^{m} \max\left\{0, f_{i}(\mathbf{x}) - r_i\right\} & \text{otherwise.}
    \end{cases}
    \end{array}
\]

Figure~\ref{fig:contour_dom} shows the level sets of the \(R^\text{Do}_{Y}\) formulation. Note that the ``zone of interest'' is larger than the union of the zones relative to each element of the reference set \(Y\). However, this formulation does not preserve the differentiability of the original problem (\ref{ref:MOP}). Also, the evaluation of \(\phi_{Y}^\text{Do}\) at a given solution has a complexity cost of \(\mathcal{O}\left(m |Y|\right)\), which could be expensive if \(Y\) is large.

\begin{figure}[!ht]
\centering
\includegraphics[width=0.7\linewidth]{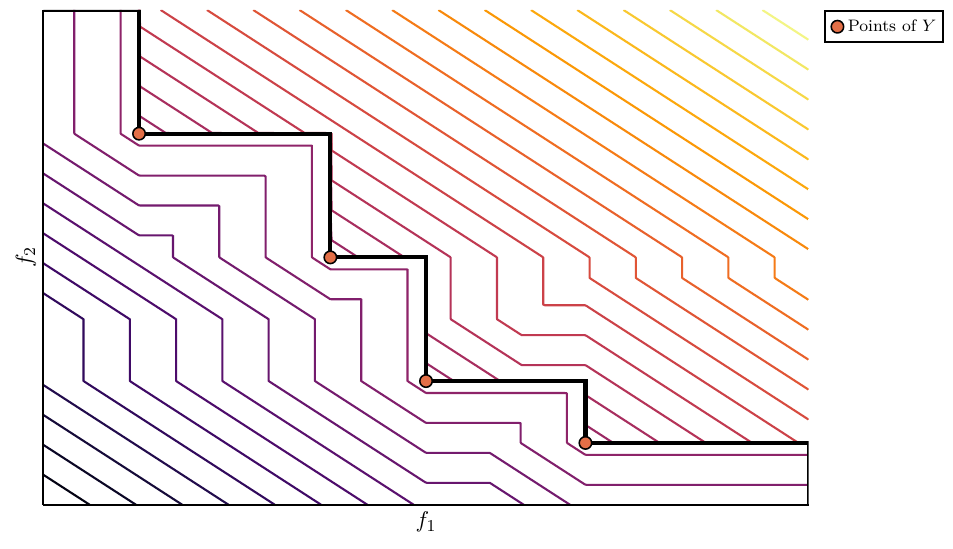}
\caption{Levels sets in the objective space of the \(R^\text{Do}_Y\) formulation for a biobjective minimization problem.}
\label{fig:contour_dom}
\end{figure}

The next theorem shows that \(R^\text{Do}_Y\) is a single-objective formulation of (\ref{ref:MOP}). Before proving the theorem, the following lemma is provided.

\begin{lemma}\label{lemma:max_min_lipschitz}
    Let \((g_i)_{i \in \mathcal{I}}\) be a finite collection of functions from \(\Omega\) to \(\mathbb{R}\). If \(g_i\) for \(i \in \mathcal{I}\) is Lipschitz continuous near some \(\bar{\mathbf{x}} \in \Omega\), then \(\displaystyle\max_{i \in \mathcal{I}} g_i\) and \(\displaystyle\min_{i \in \mathcal{I}} g_i\) are Lipschitz continuous near \(\bar{\mathbf{x}} \in \Omega\).
\end{lemma}
\begin{proof}
Let \((g_i)_{i \in I}\) be Lipschitz continuous near \(\bar{\mathbf{x}} \in \Omega\) with \(\lambda_i \geq 0\) the Lipschitz constant associated to \(g_i\) for \(i \in \mathcal{I}\).
Let \(\mathbf{x} \in \Omega\) be another point in the neighbourhood of \(\bar{\mathbf{x}} \in \Omega\), \(i_{0} \in \arg \displaystyle\max_{i \in \mathcal{I}} g_i(\bar{\mathbf{x}})\) and \(i_{1} \in \arg \displaystyle\max_{i \in \mathcal{I}} g_i(\mathbf{x})\). Then for all \(i \in \mathcal{I}\):
\[\left|\max_{i \in \mathcal{I}} g_i(\bar{\mathbf{x}}) - \max_{i \in \mathcal{I}} g_i(\mathbf{x})\right| = \begin{cases} g_{i_0}(\bar{\mathbf{x}}) - g_{i_1}(\mathbf{x}) & \leq g_{i_0}(\bar{\mathbf{x}}) - g_{i_0}(\mathbf{x}) \text{ or } \\
g_{i_1}(\mathbf{x}) - g_{i_0}(\bar{\mathbf{x}}) & \leq g_{i_1}(\mathbf{x}) - g_{i_1}(\bar{\mathbf{x}}).
\end{cases}
\]
Hence
\[\left|\max_{i \in \mathcal{I}} g_i(\bar{\mathbf{x}}) - \max_{i \in \mathcal{I}} g_i(\mathbf{x})\right| \leq \max_{i \in \mathcal{I}} \left|g_i(\bar{\mathbf{x}}) - g_i(\mathbf{x})\right| \leq \max_{i \in \mathcal{I}} \{\lambda_i\} \|\bar{\mathbf{x}} - \mathbf{x}\|.\]
Consequently, \(\displaystyle\max_{i \in \mathcal{I}} g_i\) is Lipschitz continuous near \(\bar{\mathbf{x}}\).

To prove the second part of the lemma, observe that
\[\left|\min_{i \in \mathcal{I}} g_i(\bar{\mathbf{x}}) - \min_{i \in \mathcal{I}} g_i(\mathbf{x})\right| = \left| \max_{i \in \mathcal{I}} - g_i(\bar{x}) - \max_{i \in \mathcal{I}} -g_i(\mathbf{x})\right| \leq \max_{i \in \mathcal{I}} \left|g_i(\bar{\mathbf{x}}) - g_i(\mathbf{x})\right| \leq \max_{i \in \mathcal{I}} \{\lambda_i\} \|\bar{\mathbf{x}} - \mathbf{x}\|,\]
which concludes the proof.
\end{proof}

Next, \(R_Y^\text{Do}\)'s single-objective formulation is established.

\begin{theorem} \(R^\text{Do}_Y\) is a single-objective formulation in the sense of Definition~\ref{def:single_objective_formulation}.
\end{theorem}
\begin{proof}
    Let \(f\) be Lipschitz continuous near \(\bar{\mathbf{x}} \in \Omega\).
    
    Then for all \(\mathbf{r} \in Y\), \(\mathbf{x} \mapsto r_i - f_i(\mathbf{x})\) is Lipschitz continuous near \(\bar{\mathbf{x}} \in \Omega\) for \(i = 1, 2, \ldots, m\). Consequently, using Lemma~\ref{lemma:max_min_lipschitz}, \(\mathbf{x} \mapsto \max \left\{0, r_i - f_i(\mathbf{x})\right\}\) is Lipschitz continuous near \(\bar{\mathbf{x}} \in \Omega\), for all \(\mathbf{r} \in Y\) and \(i = 1, 2, \ldots, m\).
    
    As a finite sum of Lipschitz continuous functions near \(\bar{\mathbf{x}} \in \Omega\), one can deduce that \(\mathbf{x} \mapsto \displaystyle\sum_{i = 1}^m \max \left\{0, r_i - f_i(\mathbf{x})\right\}\) is Lipschitz continuous near \(\bar{\mathbf{x}} \in \Omega\) for all \(\mathbf{r} \in Y\). 
    
    Reusing Lemma~\ref{lemma:max_min_lipschitz} yields that \(\psi^{\text{Do}, 1}_Y : \mathbf{x} \mapsto \displaystyle\min_{\mathbf{r} \in Y} \displaystyle\sum_{i = 1}^m \max \left\{0, r_i - f_i(\mathbf{x})\right\}\) is Lipschitz continuous near \(\bar{\mathbf{x}} \in \Omega\). Using the same reasoning, \(\psi^{\text{Do}, 2}_Y : \mathbf{x} \mapsto \displaystyle\min_{\mathbf{r} \in Y} \displaystyle\sum_{i = 1}^m \max \left\{0, f_i(\mathbf{x}) - r_i\right\}\) is Lipschitz continuous near \(\bar{\mathbf{x}} \in \Omega\).

    To show that \(\psi^\text{Do}_{Y}\) is Lipschitz continuous near \(\bar{\mathbf{x}} \in \Omega\), let \(\mathbf{x} \in \Omega\) be in the neighbourhood of \(\bar{\mathbf{x}} \in \Omega\). Then:
    \[\left|\psi^\text{Do}_{Y}(\bar{\mathbf{x}}) - \psi^\text{Do}_{Y}(\mathbf{x})\right| \leq \left|\max \left\{\psi^{\text{Do}, 1}_Y(\bar{\mathbf{x}}), \psi^{\text{Do}, 2}_Y(\bar{\mathbf{x}})\right\} - \max \left\{\psi^{\text{Do}, 1}_Y(\mathbf{x}), \psi^{\text{Do}, 2}_Y(\mathbf{x})\right\}\right|.\]
    Using Lemma~\ref{lemma:max_min_lipschitz}, \(\mathbf{x} \mapsto \max \left\{\psi^{\text{Do}, 1}_Y(\mathbf{x}), \psi^{\text{Do}, 2}_Y(\mathbf{x})\right\}\) is Lipschitz continuous near \(\bar{\mathbf{x}} \in \Omega\) and hence \(\psi^\text{Do}_Y\) satisfies the first condition of Definition~\ref{def:single_objective_formulation}.

    To prove the second condition of Definition~\ref{def:single_objective_formulation}, let \(f\) be Lipschitz continuous near \(\bar{\mathbf{x}} \in \Omega\) such that there exists at least one \(\mathbf{r} \in Y\) with \(f(\mathbf{x}) < \mathbf{r}\). Then \(\psi^\text{Do}_Y(\bar{\mathbf{x}}) = \sum_{i = 1}^m {f_i(\bar{\mathbf{x}}) - \tilde{r}_i}\) for one particular \(\tilde{\mathbf{r}} \in Y\). Let \(\mathbf{d} \in \mathcal{T}^\text{Cl}_{\Omega}(\bar{\mathbf{x}})\) satisfying \(f^\text{o}_i(\bar{\mathbf{x}}; \mathbf{d}) < 0\) for \(i = 1, 2, \ldots, m\). Then, by Proposition~\cite[2.3.3]{Clar83a},
    \[\psi^\text{Do}_Y(\bar{\mathbf{x}}; \mathbf{d}) \leq \sum_{i = 1}^m f^\text{o}_i(\bar{\mathbf{x}}; \mathbf{d}) < 0,\]
    which concludes the proof.
\end{proof}

All elements of the reference set \(Y \neq \emptyset\) can be selected anywhere in the objective space \(\mathbb{R}^m\), as long as their number is finite and that they are incomparable to each other. When these conditions are satisfied, one can show that the single-objective dominance move formulation \(R^\text{Do}_Y\) ``partially captures'' the Pareto dominance relation. This is discussed in the following lemma.

\begin{lemma}~\label{lem:formulation_dominance_prop} Let \((\mathbf{x}^1, \mathbf{x}^2) \in \Omega^2\) and \(Y \subset \mathbb{R}^m\) satisfying Definition~\ref{def:single_objective_formulation}.
\begin{enumerate}[label={(\roman*)}]
    \item If \(\mathbf{x}^1 \prec \prec \mathbf{x}^2\), then \(\psi^\text{Do}_Y(\mathbf{x}^1) < \psi^\text{Do}_Y(\mathbf{x}^2)\).    
    \item If \(\mathbf{x}^1 \prec \mathbf{x}^2\), then \(\psi^\text{Do}_Y(\mathbf{x}^1) \leq \psi^\text{Do}_Y(\mathbf{x}^2)\).    
\end{enumerate}
\end{lemma}
\begin{proof} Let start with the first case. Suppose that \(\mathbf{x}^1 \prec \prec \mathbf{x}^2\) and consider the following situations:
\begin{enumerate}
    \item \(f(\mathbf{x}^1)\) is not dominated by any element of \(Y\) and there exists at least one \(\mathbf{r} \in Y\) such that \(\mathbf{r} \leq f(\mathbf{x}^2)\). By definition, \(\psi^\text{Do}_Y(\mathbf{x}^1) < 0 \leq \psi^\text{Do}_Y(\mathbf{x}^2)\).
    \item \(f(\mathbf{x}^1)\) and \(f(\mathbf{x}^2)\) are dominated by at least one element of \(Y\). Let \(\mathbf{r} \in Y\). Then for all \(i = 1, 2, \ldots, m\),
    \[f_i(\mathbf{x}^1) < f_i(\mathbf{x}^2) \Longleftrightarrow f_i(\mathbf{x}^1) - r_i < f_i(\mathbf{x}^2) - r_i \Rightarrow \max \left\{0, f_i(\mathbf{x}^1) - r_i\right\} \leq \max \left\{0, f_i(\mathbf{x}^2) - r_i\right\}.\]
    All elements of \(Y\) being non-dominated between each other, there exists at least one objective index \(i_0 \in \{1, 2, \ldots, m\}\) such that \(f_{i_0}(\mathbf{x}^1) - r_{i_0} > 0 \Rightarrow f_{i_0}(\mathbf{x}^2) - r_{i_0} > f_{i_0}(\mathbf{x}^1) - r_{i_0} > 0\) (by assumption). Hence,
    \[\sum_{i = 1}^{m} \max\left\{0, f_i(\mathbf{x}^1) - r_i\right\} < \sum_{i = 1}^{m} \max \left\{0, f_i(\mathbf{x}^2) - r_i\right\}.\]
    This inequality is satisfied for all \(\mathbf{r} \in Y\). Let \((j_1, j_2) \in \{1, 2, \ldots, |Y|\}^2\) such that
    \[\psi_Y^\text{Do}(\mathbf{x}^1) = \min_{\mathbf{r} \in Y} \sum_{i = 1}^{m} \max\left\{0, f_i(\mathbf{x}^1) - r_i\right\} = \sum_{i = 1}^{m} \max \left\{0, f_i(\mathbf{x}^1) - r_i^{j_1}\right\},\]
    and,
    \[\psi_Y^\text{Do}(\mathbf{x}^2) = \min_{\mathbf{r} \in Y} \sum_{i = 1}^{m} \max \left\{0, f_i(\mathbf{x}^2) - r_i\right\} = \sum_{i = 1}^{m} \max \left\{0, f_i(\mathbf{x}^2) - r_i^{j_2}\right\}.\]
    Then:
    \[\begin{split}\psi_Y^\text{Do}(\mathbf{x}^1) = \sum_{i = 1}^{m} \max \left\{0, f_i(\mathbf{x}^1) -r_i^{j_1}\right\} \leq \sum_{i = 1}^{m} \max \left\{0, f_i(\mathbf{x}^1) - r_i^{j_2}\right\} < & \sum_{i = 1}^{m} \max \left\{0, f_i(\mathbf{x}^2) - r_i^{j_2}\right\} \\
    & = \psi_Y^\text{Do}(\mathbf{x}^2).\end{split}\]
    \item \(f_i(\mathbf{x}^1)\) and \(f(\mathbf{x}^2)\) are not dominated by any element of \(Y\). Let \(\mathbf{r} \in Y\). Then for all \(i = 1,2, \ldots, m\),
    \[f_i(\mathbf{x}^1) < f_i(\mathbf{x}^2) \Longleftrightarrow r_i - f_i(\mathbf{x}^1) > r_i - f_i(\mathbf{x}^2) \Rightarrow \max \left\{0, r_i - f_i(\mathbf{x}^1)\right\} \geq \max(0, r_i - f_i(\mathbf{x}^2)).\]
    As \(f(\mathbf{x}^2)\) is not dominated by \(\mathbf{r}\) and \(\mathbf{x}^1\) strictly dominates \(\mathbf{x}^2\), there exists at least one component objective index \(i_0 \in \{1, 2, \ldots, m\}\) such that \(r_{i_0} - f_{i_0}(\mathbf{x}^1) > r_{i_0} - f_{i_0}(\mathbf{x}^2) \geq 0\). Hence
    \[\sum_{i = 1}^{m} \max\left\{0, r_i - f_i(\mathbf{x}^1)\right\} > \sum_{i = 1}^{m} \max \left\{0, r_i - f_i(\mathbf{x}^2)\right\}.\]
    Let \((j_1, j_2) \in \{1, 2, \ldots, |Y|\}^2\) such that
    \[\min_{\mathbf{r} \in Y} \sum_{i = 1}^{m} \max \left\{0, r_i - f_i(\mathbf{x}^1)\right\} = \sum_{i = 1}^{m} \max\left\{0, r_i^{j_1} - f_i(\mathbf{x}^1)\right\},\]
    and,
    \[\min_{\mathbf{r} \in Y} \sum_{i = 1}^{m} \max\left\{0, r_i - f_i(\mathbf{x}^2)\right\} = \sum_{i = 1}^{m} \max\left\{0, r_i^{j_2} - f_i(\mathbf{x}^2))\right\}.\]
    Then:
    \[\begin{split}
    \psi_Y^\text{Do}(\mathbf{x}^1) = - \sum_{i = 1}^{m} \max\left\{0, r_i^{j_1} - f_i(\mathbf{x}^1)\right\} < - \sum_{i = 1}^{m} \max\left\{0, r_i^{j_1} - f_i(\mathbf{x}^2)\right\} \leq & - \sum_{i = 1}^{m} \max \left\{0, r_i^{j_2} - f_i(\mathbf{x}^2)\right\} \\
    & = \psi_Y^\text{Do}(\mathbf{x}^2).  
    \end{split}\]
\end{enumerate}
The second part of the lemma follows the same reasoning. Note that in this case, the inequality is not strict in Cases \(2\) and \(3\), as illustrated in Figure~\ref{fig:counter-examples_for_psi_equality}.
\end{proof}

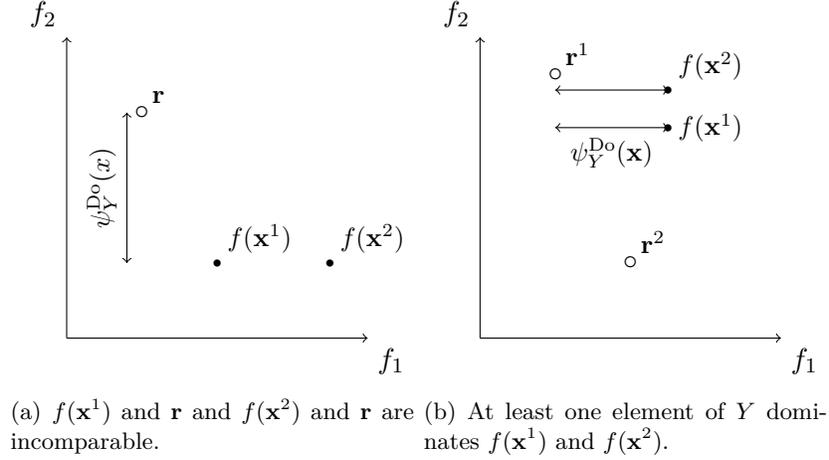
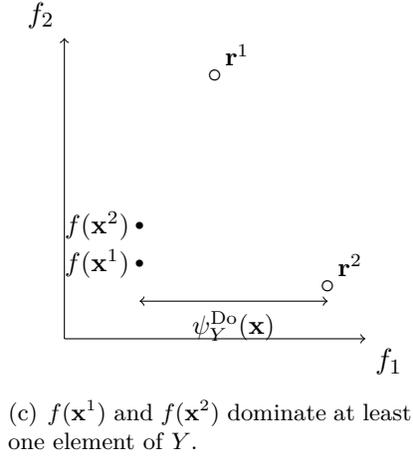
\begin{figure}[!th]
      \centering
      \subfigure[\(f(\mathbf{x}^1)\) and \(\mathbf{r}\) and \(f(\mathbf{x}^2)\) and \(\mathbf{r}\) are incomparable.]{
      \begin{tikzpicture}
        % Axes
        \draw [->] (0,0) -- (4,0) node[below right]{\(f_1\)};
		\draw [->] (0,0) -- (0,4) node[above left]{\(f_2\)};

        % Points
        \draw (1, 3) node{\(\circ\)} node[above right]{\small{\(\mathbf{r}\)}};
        \draw (2, 1) node{\tiny{\(\bullet\)}} node[above right]{\small{\(f(\mathbf{x}^1)\)}};
        \draw (3.5, 1) node{\tiny{\(\bullet\)}} node[above right]{\small{\(f(\mathbf{x}^2)\)}};

        % Psi
        \draw[<->] (0.8, 1) -- (0.8, 3) node[midway, sloped, above]{\small{\(\psi_Y^\text{Do}(x)\)}}; 
      \end{tikzpicture}}
       \subfigure[At least one element of \(Y\) dominates \(f(\mathbf{x}^1)\) and \(f(\mathbf{x}^2)\).]{
      \begin{tikzpicture}
        % Axes
        \draw [->] (0,0) -- (4,0) node[below right]{\(f_1\)};
		\draw [->] (0,0) -- (0,4) node[above left]{\(f_2\)};

        % Points
        \draw (1, 3.5) node{\(\circ\)} node[above right]{\small{\(\mathbf{r}^1\)}};
        \draw (2, 1) node{\(\circ\)} node[above right]{\small{\(\mathbf{r}^2\)}};
        \draw (2.5, 2.8) node{\tiny{\(\bullet\)}} node[right]{\small{\(f(\mathbf{x}^1)\)}};
        \draw (2.5, 3.3) node{\tiny{\(\bullet\)}} node[above right]{\small{\(f(\mathbf{x}^2)\)}};

        % Psi
        \draw[<->] (1, 3.3) -- (2.5, 3.3);
        \draw[<->] (1, 2.8) -- (2.5, 2.8) node[midway, below]{{\small{\(\psi_Y^\text{Do}(\mathbf{x})\)}}};
      \end{tikzpicture}}
      \subfigure[\(f(\mathbf{x}^1)\) and \(f(\mathbf{x}^2)\) dominate at least one element of \(Y\).]{
      \begin{tikzpicture}
        % Axes
        \draw [->] (0,0) -- (4,0) node[below right]{\(f_1\)};
		\draw [->] (0,0) -- (0,4) node[above left]{\(f_2\)};

        % Points
        \draw (2, 3.5) node{\(\circ\)} node[above right]{\small{\(\mathbf{r}^1\)}};
        \draw (3.5, 0.7) node{\(\circ\)} node[above right]{\small{\(\mathbf{r}^2\)}};
        \draw (1, 1) node{\tiny{\(\bullet\)}} node[left]{\small{\(f(\mathbf{x}^1)\)}};
        \draw (1, 1.5) node{\tiny{\(\bullet\)}} node[left]{\small{\(f(\mathbf{x}^2)\)}};

        % Psi
        \draw[<->] (1, 0.5) -- (3.5, 0.5) node[midway, below]{{\small{\(\psi_Y^\text{Do}(\mathbf{x})\)}}};
      \end{tikzpicture}}
      \caption{Illustration for a biobjective minimization problem: \(\mathbf{x}^1 \prec \mathbf{x}^2\) but \(\psi_Y^\text{Do}(\mathbf{x}^1) = \psi_Y^\text{Do}(\mathbf{x}^2)\).}
      \label{fig:counter-examples_for_psi_equality}
  \end{figure}

Finally, the following theorem shows that solving the \(R_Y^\text{Do}\) single-objective formulation could result in a locally Pareto optimal solution. Note that, in practice, one does not solve \(R^\text{Do}_Y\) to optimality.

\begin{theorem}[adapted from~\cite{AuSaZg2010a}] If \(\hat{\mathbf{x}} \in \Omega\) is the unique solution of \(R^\text{Do}_Y\) for some reference set \(Y \subset \mathbb{R}^m\) satisfying Definition~\ref{def:single_objective_formulation}, then \(\bar{\mathbf{x}}\) is Pareto optimal for (\ref{ref:MOP}). 
\end{theorem}
\begin{proof}
    Let \(\hat{\mathbf{x}} \in \Omega\) be the unique solution of \(R^\text{Do}_Y\). Let \(\mathbf{x} \in \Omega\) with \(\mathbf{x} \neq \hat{\mathbf{x}}\). Then \(\psi^\text{Do}_Y(\hat{\mathbf{x}}) < \psi^\text{Do}_Y(\mathbf{x})\). By Lemma~\ref{lem:formulation_dominance_prop} and the fact that \(\mathbf{x} \neq \hat{\mathbf{x}}\), there exists at least one index \(i_0 \in \{1, 2, \ldots, m\}\) such that \(f_{i_0}(\hat{\mathbf{x}}) < f_{i_0}(\mathbf{x})\) and \(\mathbf{x}\) does not dominate \(\hat{\mathbf{x}}\). The point \(\hat{\mathbf{x}}\) is therefore Pareto optimal.
\end{proof}

% ----------------------------------------------------%
\section{Search strategies for the DMulti-MADS algorithm}
\label{sect:Search steps}
% ----------------------------------------------------%

When designing an algorithm that solves a multiobjective problem, several goals are considered to obtain a good Pareto set approximation.
\begin{itemize}
    \item Ideally, all elements of the Pareto set approximation should be (locally) Pareto optimal, i.e., the distance between the Pareto front and its representation in the objective space should be minimized.
    \item It should contain the extreme solutions of the Pareto set, i.e., solutions that minimize each objective component separately. Such solutions provide information about the ``limit'' trade-offs the decision maker can face when a choice has to be made.
    \item Its representation in the objective space should be dense, i.e., the gaps between the different solutions should be minimized.
\end{itemize}

Originally, DMulti-MADS tries to satisfy these different goals by using the information provided by the frame size parameters of the iterate list. Their values act as indicators to quantify the progression of each solution towards a locally Pareto optimal point. Furthermore, DMulti-MADS takes into account the objective values of each element of the iterate list to fulfill largest gaps into the current Pareto front approximation~\cite{BiLedSa2020}. However, restricting to poll steps only as it was proposed in the implementation of DMulti-MADS may be too rigid in exploring the decision space and not precise enough to efficiently achieve the three goals mentioned above.

This section describes two search strategies for DMulti-MADS, based on the single-objective formulations described in Section~\ref{sect:Single-objective reformulations}.

\subsection{The MultiMADS search strategy}\label{subsec:MultiMADS_search_strategy}

Inspired by the BiMADS and MultiMADS algorithms, the MultiMADS search strategy introduces an additional search step in DMulti-MADS that aims to accelerate the algorithm progress towards a good solution set. This search adaptively switches between a single-objective distance formulation subproblem optimization presented in Section~\ref{sect:Single-objective reformulations} and an extreme point exploration, depending on the nature of the current frame incumbent \(\mathbf{x}^k\) at iteration \(k\). Algorithm~\ref{alg:MultiMADS_search_strategy} details the procedure for a fixed \(k\).

\begin{figure}[!th]
  \begin{algorithm}[H]\small
   \caption{The MultiMADS search strategy.}
   \begin{algorithmic}[1]
    \STATE \textbf{Initialization}: Let \((\mathbf{x}^k, \Delta^k) \in L^k\) be the current frame incumbent chosen by DMulti-MADS at iteration \(k\) and \(V^k \subset \mathcal{X}\) the set of points already evaluated by DMulti-MADS from the beginning of iteration \(k\).
    \STATE \textbf{Determination of the nature of \(\mathbf{x}^k\)}. If \(\mathbf{x}^k\) is a current extreme solution, i.e., there exists at least one objective index \(i_0 \in \{1, 2, \ldots, m\}\) for which \(f_{i_0}(\mathbf{x}^k) = \displaystyle\min_{(\mathbf{x}, \Delta) \in L^k} f_{i_0}(\mathbf{x})\), go to \textit{extreme point exploration}; else go to \textit{single-objective distance formulation exploration}.
    \STATE \textbf{Extreme point exploration}.
    \FOR{objective index \(i \in \{1, 2, \ldots, m\}\) for which \(\mathbf{x}^k\) is a current extreme solution}
    \STATE Solve
    \[\min_{\mathbf{x} \in \mathcal{X}^k} f_i(\mathbf{x}),\]
    using a subsolver procedure, starting from \(\mathbf{x}^k\), where \(\mathcal{X}^k \subseteq \mathcal{X}\) is a subset of \(X\), depending on \(\Delta^k\), \(V^k\) and the choice of the subsolver.
    \STATE Collect all solution points \(S^\text{tmp} \subset \mathcal{X} \cap M^k\) proposed by the subsolver. Set \(S^k := S^k \cup S^\text{tmp}\).
    \ENDFOR
    \STATE If a success criterion is met, the search procedure may terminate.
    \STATE Exit the procedure.
    \STATE \textbf{Single-objective distance formulation exploration.} Compute the reference objective vector
    \[\mathbf{r}^k := \texttt{computeSingleObjectiveDistanceFormReferencePoint}(\mathbf{x}^k, L^k) \text{ (see Algorithm~\ref{alg:compute_single_obj_distance_form_reference_pt})}.\]
    Solve
    \[\min_{\mathbf{x} \in \mathcal{X}^k} \bar{\psi}_{\mathbf{r}^k}(\mathbf{x}),\]
    using a subsolver procedure, starting from \(\mathbf{x}^k\), where \(\mathcal{X}^k \subseteq \mathcal{X}\) is a subset of \(\mathcal{X}\), depending on \(\Delta^k\), \(V^k\) and the choice of the subsolver.
    \STATE Collect all solution points \(S^\text{tmp} \subset \mathcal{X} \cap M^k\) proposed by the subsolver. Set \(S^k := S^k \cup S^\text{tmp}\). If a success criterion is met, the search procedure may terminate.
   \end{algorithmic}
   \label{alg:MultiMADS_search_strategy}
  \end{algorithm}
\end{figure}

  The procedure starts by determining the nature of the current incumbent \(\mathbf{x}^k\). If it is a current extreme solution, an attempt is made to extend the current Pareto front approximation by trying to minimize each objective for which \(\mathbf{x}^k\) is marked as extreme. Otherwise, the MultiMADS search strategy computes a reference objective vector \(\mathbf{r}^k \in \mathbb{R}^m\) using the iterate list \(L^k\) and the current incumbent \(\mathbf{x}^k\).
  
  Algorithm~\ref{alg:compute_single_obj_distance_form_reference_pt} details this process. This procedure is an adaptation of the reference point determination procedure used for BiMADS~\cite{AuSaZg2008a} to more than two objectives. It returns a reference vector \(\mathbf{r}^k\) satisfying \(f(\mathbf{x}^k) \leq \mathbf{r}^k\) or \(\mathbf{r}^k = f(\mathbf{x}^k)\). Note that \(\mathbf{r}^k\) is never computed if \(\mathbf{x}^k\) is an extreme point of the current Pareto front approximation. By approximately solving the associated single-objective distance formulation subproblem, one can hope to converge towards the Pareto front, or to fill the gaps in the current Pareto front approximation around \(f(\mathbf{x}^k)\).

  \begin{figure}[!th]
  \begin{algorithm}[H]\small
   \caption{\(\texttt{computeSingleObjectiveDistanceFormReferencePoint}(\mathbf{x}^k, L^k)\)}
   \begin{algorithmic}
    \STATE Let \(\mathbf{r}^k := \mathbf{0}.\)
    \FOR{\(i = 1, 2, \ldots, m\)}
        \STATE Order \(L^k = \{(\mathbf{x}^1, \Delta^1), (\mathbf{x}^2, \Delta^2), \ldots, (\mathbf{x}^{|L^k|}, \Delta^{|L^k|}\}\) by increasing objective component value \(f_i\), i.e., \(f_i(\mathbf{x}^1) \leq f_i(\mathbf{x}^2) \leq \ldots \leq f_i(\mathbf{x}^{|L^k|})\).
        
        Let \(j_k \in \{1,\ldots, |L^k|\}\) be the list index of the sorted iterate list corresponding to \((\mathbf{x}^{j_k}, \Delta^{j_k}) = (\mathbf{x}^{k}, \Delta^{k})\).
        
        Define the following list index \(j_\text{select} := \min \left\{j_k+1, |L^k|\right\}\).

        Set \(r_i := f_i(\mathbf{x}^{j_\text{select}})\).
    \ENDFOR
    \STATE Return \(\mathbf{r}^k\).
   \end{algorithmic}
   \label{alg:compute_single_obj_distance_form_reference_pt}
  \end{algorithm}
  %\caption{The computation of the reference objective vector \(\mathbf{r}^k\) for the MultiMADS search strategy.}
  \end{figure}

  All the points proposed by the subsolver are stored into \(S^k\) and evaluated by the blackbox. When a success criterion is met, the search may terminate, i.e., the other search strategies that follow this search procedure are not executed; the poll is skipped and DMulti-MADS goes to the update step as described in Algorithm~\ref{alg:summary_DMulti-MADS_algorithm}.
  
  The efficiency of this procedure depends on the choice of the subsolver, which will be detailed in Section~\ref{subsec:subsolvers_algo}.

  Figure~\ref{fig:zones_interest_multimads_strategy} illustrates the zones of interest prioritized by the subsolver for different choices of \(\mathbf{x}^k\).

  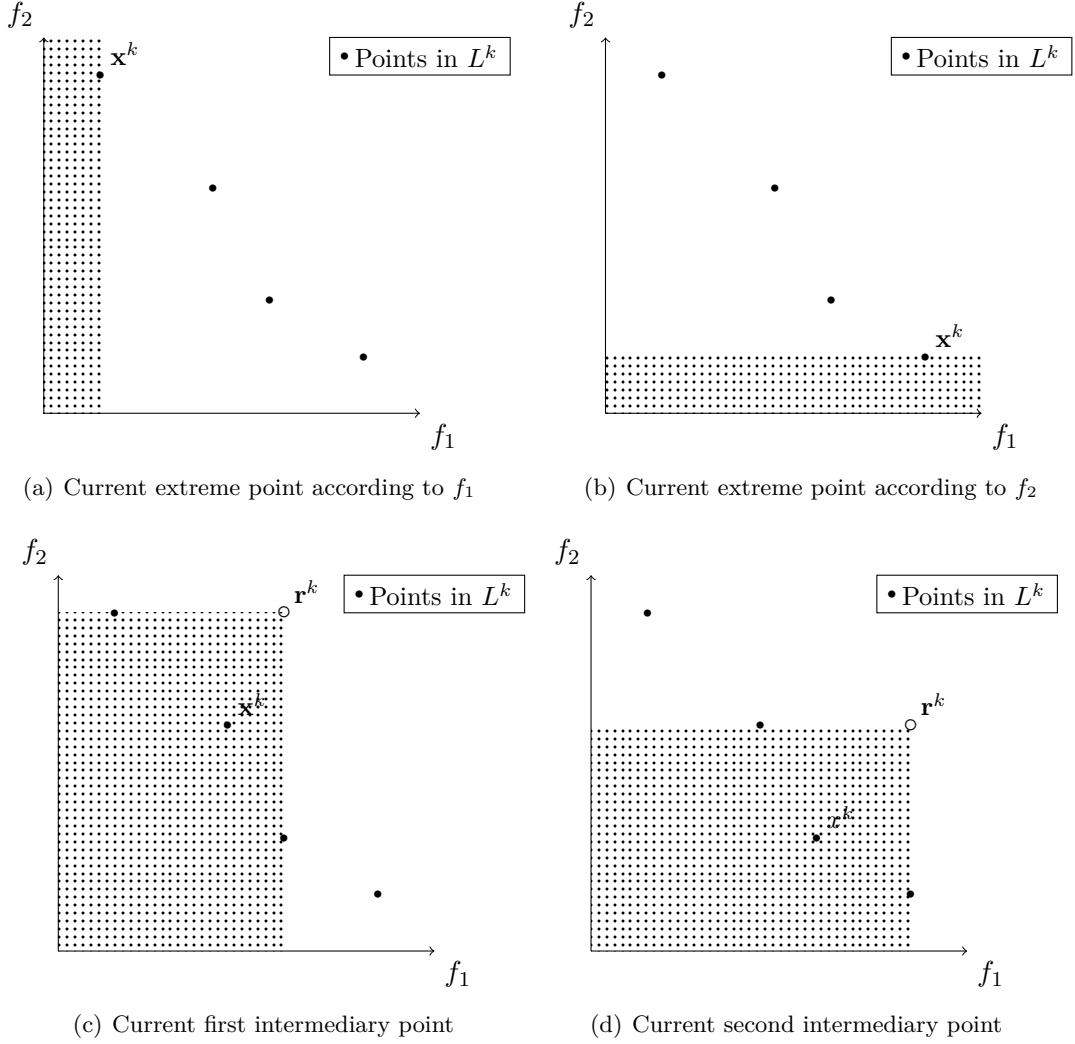
\begin{figure}[!ht]
      \centering
      \subfigure[Current extreme point according to \(f_1\)]{
      \begin{tikzpicture}
        % Axes
        \draw [->] (0,0) -- (5,0) node[below right]{\(f_1\)};
		\draw [->] (0,0) -- (0,5) node[above left]{\(f_2\)};

        % List L^k
        \draw (0.75, 4.5) node{\tiny{\(\bullet\)}} node[above right]{\small{\(\mathbf{x}^k\)}};
        \draw (2.25, 3) node{\tiny{\(\bullet\)}};
        \draw (3, 1.5) node{\tiny{\(\bullet\)}};
        \draw (4.25, 0.75) node{\tiny{\(\bullet\)}};

        % Zones of interest
        \fill [pattern=dots, pattern color=black] (0, 0) rectangle (0.75, 5);

        % Caption
        \draw (4, 4.75) node{{\tiny \(\bullet\)}} node[right] {\small{Points in \(L^k\)}};
        \draw (3.8, 4.5) rectangle (6.2, 5);
      \end{tikzpicture}}
      \quad
      \subfigure[Current extreme point according to \(f_2\)]{
      \begin{tikzpicture}
        % Axes
        \draw [->] (0,0) -- (5,0) node[below right]{\(f_1\)};
		\draw [->] (0,0) -- (0,5) node[above left]{\(f_2\)};

        % List L^k
        \draw (0.75, 4.5) node{\tiny{\(\bullet\)}};
        \draw (2.25, 3) node{\tiny{\(\bullet\)}};
        \draw (3, 1.5) node{\tiny{\(\bullet\)}};
        \draw (4.25, 0.75) node{\tiny{\(\bullet\)}} node[above right]{\small{\(\mathbf{x}^k\)}};

        % Zones of interest
        \fill [pattern=dots, pattern color=black] (0, 0) rectangle (5, 0.75);

        % Caption
        \draw (4, 4.75) node{{\tiny \(\bullet\)}} node[right] {\small{Points in \(L^k\)}};
        \draw (3.8, 4.5) rectangle (6.2, 5);
      \end{tikzpicture}}
      \quad
      \subfigure[Current first intermediary point]{
      \begin{tikzpicture}
        % Axes
        \draw [->] (0,0) -- (5,0) node[below right]{\(f_1\)};
		\draw [->] (0,0) -- (0,5) node[above left]{\(f_2\)};

        % List L^k
        \draw (0.75, 4.5) node{\tiny{\(\bullet\)}};
        \draw (2.25, 3) node{\tiny{\(\bullet\)}} node[above right]{\small{\(\mathbf{x}^k\)}};
        \draw (3, 1.5) node{\tiny{\(\bullet\)}};
        \draw (4.25, 0.75) node{\tiny{\(\bullet\)}};

        % Reference point r^k
        \draw (3, 4.5) node{\(\circ\)} node[above right]{\small{\(\mathbf{r}^k\)}};;

        % Zones of interest
        \fill [pattern=dots, pattern color=black] (0, 0) rectangle (3, 4.5);

        % Caption
        \draw (4, 4.75) node{{\tiny \(\bullet\)}} node[right] {\small{Points in \(L^k\)}};
        \draw (3.8, 4.5) rectangle (6.2, 5);
      \end{tikzpicture}}
      \subfigure[Current second intermediary point]{
      \begin{tikzpicture}
        % Axes
        \draw [->] (0,0) -- (5,0) node[below right]{\(f_1\)};
		\draw [->] (0,0) -- (0,5) node[above left]{\(f_2\)};

        % List L^k
        \draw (0.75, 4.5) node{\tiny{\(\bullet\)}};
        \draw (2.25, 3) node{\tiny{\(\bullet\)}};
        \draw (3, 1.5) node{\tiny{\(\bullet\)}} node[above right]{\small{\(x^k\)}};
        \draw (4.25, 0.75) node{\tiny{\(\bullet\)}};

        % Reference point r^k
        \draw (4.25, 3) node{\(\circ\)} node[above right]{\small{\(\mathbf{r}^k\)}};;

        % Zones of interest
        \fill [pattern=dots, pattern color=black] (0, 0) rectangle (4.25, 3);

        % Caption
        \draw (4, 4.75) node{{\tiny \(\bullet\)}} node[right] {\small{Points in \(L^k\)}};
        \draw (3.8, 4.5) rectangle (6.2, 5);
      \end{tikzpicture}}
      \caption{Zones of interest (dotted) in the objective space for the MultiMADS search strategy for different configuration of incumbent points \(\mathbf{x}^k\) in an iterate list \(L^k\) for a biobjective minimization problem.}
      \label{fig:zones_interest_multimads_strategy}
  \end{figure}

\subsection{The dominance move search strategy}\label{subsec:DoM_search_strategy}

Algorithm~\ref{alg:DoM_search_strategy} describes the dominance move search strategy. The procedure starts by determining a reference objective set, using the iterate list \(L^k\) and the current incumbent solution \(\mathbf{x}^k\). If the iterate list contains a single element, \(f(\mathbf{x}^k)\) is used as the reference set. Otherwise, the reference set contains all objective vectors \(f(\mathbf{x})\) corresponding to each element \((\mathbf{x}, \Delta)\) of the iterate list \(L^k\), at the exception of \(f(\mathbf{x}^k)\). Then, the procedure solves the corresponding single-objective dominance formulation subproblem using a single-objective subsolver (see Section~\ref{subsec:subsolvers_algo}).

\begin{figure}[!th]
  \begin{algorithm}[H]\small
   \caption{The dominance move search strategy}
   \begin{algorithmic}[1]
    \STATE \textbf{Initialization}: Let \((\mathbf{x}^k, \Delta^k) \in L^k\) be the current frame incumbent chosen by DMulti-MADS (see~Algorithm~\ref{alg:summary_DMulti-MADS_algorithm}) at iteration \(k\) and \(V^k \subset \mathcal{X}\) the set of points already evaluated by DMulti-MADS since the beginning of iteration \(k\).
    \STATE \textbf{Determination of the reference objective list \(R^{k}\)}. If \(|L^k| = 1\), set \(R^k := \{f(\mathbf{x}^k)\}\). Otherwise, set \(R^k := \{f(\mathbf{x}): (\mathbf{x}, \Delta) \in L^k\} \setminus \{f(\mathbf{x}^k)\}\).
    \STATE \textbf{Single-objective dominance move exploration}.
    \STATE Solve
    \[\min_{\mathbf{x} \in \mathcal{X}^k} \psi^\text{Do}_{R^k}(\mathbf{x})\]
    using a subsolver procedure, starting from \(\mathbf{x}^k\), where \(\mathcal{X}^k \subseteq \mathcal{X}\) is a subset of \(\mathcal{X}\), depending on \(\Delta^k\), \(V^k\) and the choice of the subsolver.
    Collect all solution points \(S^\text{tmp} \subset \mathcal{X} \cap M^k\) proposed by the subsolver. Set \(S^k := S^k \cup S^\text{tmp}\).
   \end{algorithmic}
   \label{alg:DoM_search_strategy}
  \end{algorithm}
  %\caption{The dominance move search strategy.}
\end{figure}

All the points proposed by the subsolver are stored in the search set \(S^k\). As for the MultiMADS search strategy, the search may terminate if a success criterion is met. In this case, DMulti-MADS skips all the successive search strategies, and the poll step, and moves to the parameter update (see Algorithm~\ref{alg:summary_DMulti-MADS_algorithm}).

If the dominance move search strategy considers all elements of the iterate list \(L^k\) for the construction of the reference objective set \(R^k\), then for all (\(\mathbf{x}, \Delta) \in L^k\), \(\psi^\text{Do}_{R^k}(\mathbf{x}) = 0\). By excluding \(f(\mathbf{x}^k)\) from \(R^k\), the search strategy prioritizes the exploration of a zone near \(f(\mathbf{x}^k)\),  while retaining the possibility for the subsolver to search for other points in non-dominated parts by \(R^k\) of the objective space. Figure~\ref{fig:zones_interest_DoM_strategy} illustrates the zones of interest in the objective space for a biobjective minimization problem.

  \begin{figure}[!ht]
      \centering
      \subfigure[]{
      \begin{tikzpicture}
        % Axes
        \draw [->] (0,0) -- (5,0) node[below right]{\(f_1\)};
		\draw [->] (0,0) -- (0,5) node[above left]{\(f_2\)};

        % List L^k
        \draw (0.75, 4.5) node{\tiny{\(\bullet\)}} node[above right]{\small{\(\mathbf{x}^k\)}};
        \draw (2.25, 3) node{\tiny{\(\bullet\)}};
        \draw (3, 1.5) node{\tiny{\(\bullet\)}};
        \draw (4.25, 0.75) node{\tiny{\(\bullet\)}};

        % Zones of interest
        \fill [pattern=dots, pattern color=black] (0, 0) -- (5, 0) -- (5, 0.75) -- (4.25, 0.75) -- (4.25, 1.5) -- (3, 1.5) -- (3, 3) -- (2.25, 3) -- (2.25, 5) -- (0, 5) -- cycle;

        % Caption
        \draw (4, 4.75) node{{\tiny \(\bullet\)}} node[right] {\small{Points in \(L^k\)}};
        \draw (3.8, 4.5) rectangle (6.2, 5);
      \end{tikzpicture}}
      \quad
      \subfigure[]{
      \begin{tikzpicture}
        % Axes
        \draw [->] (0,0) -- (5,0) node[below right]{\(f_1\)};
		\draw [->] (0,0) -- (0,5) node[above left]{\(f_2\)};

        % List L^k
        \draw (0.75, 4.5) node{\tiny{\(\bullet\)}};
        \draw (2.25, 3) node{\tiny{\(\bullet\)}};
        \draw (3, 1.5) node{\tiny{\(\bullet\)}};
        \draw (4.25, 0.75) node{\tiny{\(\bullet\)}} node[above right]{\small{\(\mathbf{x}^k\)}};

        % Zones of interest
        \fill [pattern=dots, pattern color=black] (0, 0) -- (5, 0) -- (5, 1.5) -- (3, 1.5) -- (3, 3) -- (2.25, 3) -- (2.25, 4.5) -- (0.75, 4.5) -- (0.75, 5) -- (0, 5) -- cycle;

        % Caption
        \draw (4, 4.75) node{{\tiny \(\bullet\)}} node[right] {\small{Points in \(L^k\)}};
        \draw (3.8, 4.5) rectangle (6.2, 5);
      \end{tikzpicture}}
      \quad
      \subfigure[]{
      \begin{tikzpicture}
        % Axes
        \draw [->] (0,0) -- (5,0) node[below right]{\(f_1\)};
		\draw [->] (0,0) -- (0,5) node[above left]{\(f_2\)};

        % List L^k
        \draw (0.75, 4.5) node{\tiny{\(\bullet\)}};
        \draw (2.25, 3) node{\tiny{\(\bullet\)}} node[above right]{\small{\(\mathbf{x}^k\)}};
        \draw (3, 1.5) node{\tiny{\(\bullet\)}};
        \draw (4.25, 0.75) node{\tiny{\(\bullet\)}};

        % Zones of interest
        \fill [pattern=dots, pattern color=black] (0, 0) -- (5, 0) -- (5, 1.5) -- (3, 1.5) -- (3, 3) -- (3, 4.5) -- (0.75, 4.5) -- (0.75, 4.5) -- (0.75, 5) -- (0, 5) -- cycle;

        % Caption
        \draw (4, 4.75) node{{\tiny \(\bullet\)}} node[right] {\small{Points in \(L^k\)}};
        \draw (3.8, 4.5) rectangle (6.2, 5);
      \end{tikzpicture}}
      \subfigure[]{
      \begin{tikzpicture}
        % Axes
        \draw [->] (0,0) -- (5,0) node[below right]{\(f_1\)};
		\draw [->] (0,0) -- (0,5) node[above left]{\(f_2\)};

        % List L^k
        \draw (0.75, 4.5) node{\tiny{\(\bullet\)}};
        \draw (2.25, 3) node{\tiny{\(\bullet\)}};
        \draw (3, 1.5) node{\tiny{\(\bullet\)}} node[above right]{\small{\(\mathbf{x}^k\)}};
        \draw (4.25, 0.75) node{\tiny{\(\bullet\)}};

        % Zones of interest
        \fill [pattern=dots, pattern color=black] (0, 0) -- (5, 0) -- (5, 0.75) -- (4.25, 0.75) -- (4.25, 3) -- (3, 3) -- (2.25, 3) -- (2.25, 4.5) -- (0.75, 4.5) -- (0.75, 5) -- (0, 5) -- cycle;

        % Caption
        \draw (4, 4.75) node{{\tiny \(\bullet\)}} node[right] {\small{Points in \(L^k\)}};
        \draw (3.8, 4.5) rectangle (6.2, 5);
      \end{tikzpicture}}
      \caption{Zones of interest (dotted) in the objective space for the dominance move search strategy for different configuration of incumbent points \(\mathbf{x}^k\) in an iterate list \(L^k\) for a biobjective minimization problem.}
      \label{fig:zones_interest_DoM_strategy}
  \end{figure}

\subsection{Subsolvers algorithms}\label{subsec:subsolvers_algo}

This section details the single-objective subsolvers used by the new search strategies presented in Sections~\ref{subsec:MultiMADS_search_strategy} and~\ref{subsec:DoM_search_strategy}.

\subsubsection{The quadratic model search step}

Quadratic models are commonly used in single-objective DFO to locally approximate blackbox functions (see for example~\cite{CoLed2011, CuVi07}). This work uses quadratic regression models inspired by~\cite{CoLed2011}, but other models could also have been considered: e.g., minimum Frobenius norm models~\cite{CuVi07}, second directional derivative-based update Hessian models~\cite{BuOlTu2015}.

Consider the following quadratic model
\[Q(\mathbf{x}) = \alpha_0 + \mathbf{g}^\top \mathbf{x} + \dfrac{1}{2} \mathbf{x}^\top \mathbf{H} \mathbf{x},\]
of a true function \(f_\text{bb}\), where \(\alpha_0 \in \mathbb{R}\), \(\mathbf{g} \in \mathbb{R}^n\) and \(\mathbf{H} \in \mathbb{R}^{n \times n}\) symmetric are the unknown parameters to identify. These \(p = \frac{(n+1) (n+2)}{2}\) parameters are determined by solving the following least-square problem with respect to the \(\left\|.\right\|_2\) norm:
\[\min_{(\alpha_0, \mathbf{g}, \mathbf{H})} \left\|Q(X_\text{sample}) - f_\text{bb}(X_\text{sample})\right\|^2_2,\]
where \(Q(X_\text{sample})\) and \(f_\text{bb}(X_\text{sample})\) are defined as:
\[Q(X_\text{sample}) = \begin{bmatrix} Q(\mathbf{x}^1) \\ Q(\mathbf{x}^2) \\ \vdots \\ Q(\mathbf{x}^q)\end{bmatrix} \text{ and } f_\text{bb}(X_\text{sample}) = \begin{bmatrix} f_\text{bb}(\mathbf{x}^1) \\ f_\text{bb}(\mathbf{x}^2) \\ \vdots \\ f_\text{bb}(\mathbf{x}^q)\end{bmatrix}\]
with \(q = |X_\text{sample}| \geq p\) and \(X_\text{sample} = \{\mathbf{x}^1, \mathbf{x}^2, \ldots, \mathbf{x}^q\}\) is a set of points that have previously been evaluated by the true function \(f_\text{bb}\).

If the sample set has good properties, i.e., the so-called \(\Lambda_R\)-poisedness condition for quadratic regression (\cite[Definition 4.7]{CoScVibook}), one can show that the error between the true function value and its quadratic regression model, as well as their gradients and Hessians is bounded by some relative constant depending on \(p\), \(\Lambda_R\) and the radius of a ball containing the sample set \(X_\text{sample}\) (\cite[Theorem 4.13]{CoScVibook}).

At iteration \(k\), the quadratic model search step tries to construct \(m + 1\) quadratic regression models: one for a given scalarization single-objective formulation \(\psi\) function and one for each \(c_j, j \in \mathcal{J}\). These models, respectively denoted as \(Q_{\psi}\) and \(Q_{c_j}, j \in \mathcal{J}\), are supposed to be local approximations of their respective true functions. In a certain region \(\mathcal{B}\), one could hope that these models satisfy:
\[Q_{\psi}(x) \approx \psi(\mathbf{x}) \text{ and } Q_{c_j}(\mathbf{x}) \approx c_j(\mathbf{x}), \text{ for all } j \in \mathcal{J}, \text{ for all } \mathbf{x} \in \mathcal{B},\]
and may be used to guide the search towards new interesting solutions.

At iteration \(k\), consider one of the two incumbent iterate solutions and denote it by \(\mathbf{x}^{k, \text{c}}\). The quadratic model search step starts by collecting a sample set of already evaluated points having finite objective values and constraints in a ball \(\mathcal{B}_{\infty}(\mathbf{x}^{k, \text{c}}, \rho \Delta^k)\), i.e.,
\[X^k_{sample} \subset \mathcal{X} \cap V^k \cap \mathcal{B}_{\infty}(\mathbf{x}^{k, \text{c}}, \rho \Delta^k),\]
where \(\rho > 1\) is a radius factor. If \(q = |X^k_\text{sample}| < p = \frac{(n+1)(n+2)}{2}\), the procedure stops, because there are not enough points to build the model.

Otherwise, the quadratic model search step tries to construct the \(m+1\) quadratic regression models using \(X^k_\text{sample}\). If the least-square resolution fails for at least one model, the procedure stops. Note that the procedure does not assess the quality of the sample set.

Finally, using these models, the quadratic model search step solves the single-objective subproblem:
\[\displaystyle\min_{\mathbf{x} \in \mathbb{R}^n} Q_\psi(\mathbf{x})
\text{ s.t. } \begin{cases} Q_{c_j}(\mathbf{x}) \leq 0, \ j \in \mathcal{J}, \\
\mathbf{x} \in \mathcal{X}, \\
\mathbf{x} \in \mathcal{B}_{\infty}(\mathbf{x}^{k, \text{c}}, \rho \Delta^k).
\end{cases}\]
The solution of these subproblems is then evaluated by the blackbox function after projection onto mesh \(M^k\). Depending on the search strategy, \(\psi\) may correspond to:
\begin{itemize}
    \item \(f_i\) for \(i = 1, 2, \ldots, m\) and \(\bar{\psi}_{r^k}\) for the MultiMADS search strategy;
    \item \(\psi^\text{Do}_{R^k}\) for the dominance move search strategy.
\end{itemize}

This subsection concludes with two remarks about the quadratic model search.

\begin{remark}
Contrary to the single-objective case~\cite{CoLed2011}, the quadratic model search only builds quadratic models around a current incumbent solution, corresponding to the primary frame center of iteration \(k\)~\cite{G-2022-10} (which may be different from the current iterate incumbent). This is done for computational time considerations because the resolution of the quadratic subproblem may be costly.
\end{remark}

\begin{remark}
When \(\psi \in \{\bar{\psi}_{r^k}, \psi^\text{Do}_{Y^k}\}\), the resulting model is not built by combining the quadratic models \(Q_{f_i}\), but directly by computing a regression on \(\psi\). Although this is less precise than the first option because structure is lost, it simplifies the resolution of the subproblem, as it is always a quadratically constrained quadratic problem with bounds on the variables. Furthermore, it improves computation time.
\end{remark}

\subsubsection{The Nelder-Mead search step}

The NM search step for multiobjective optimization follows the procedure described in the single-objective counterpart of this work~\cite{AuTr2018}. The definitions are first stated before adapting the strategy to multiobjective optimization problems. The NM strategy relies on a strict ordering of decision vectors evaluated by the algorithm.

\begin{definition}[adapted from~\cite{AuDe09a}]
    Let \(\psi : \mathbb{R}^n \rightarrow \mathbb{R} \cup \{+ \infty\}\) a scalar-valued function and two decision vectors \(\mathbf{x}^1\) and \(\mathbf{x}^2\) in \(\mathcal{X}\). The decision vector \(\mathbf{x}^1 \in \mathcal{X}\) is said to \(\psi\)-dominate the decision vector \(\mathbf{x}^2 \in \mathcal{X}\) (denoted as \(\mathbf{x}^1 \prec_{\psi} \mathbf{x}^2\)) if
    \begin{itemize}
        \item both points are feasible and \(\psi(\mathbf{x}^1) < \psi(\mathbf{x}^2)\);
        \item both points are infeasible and \(\psi(\mathbf{x}^1) \leq \psi(\mathbf{x}^2)\) and \(h(\mathbf{x}^1) \leq h(\mathbf{x}^2)\) with at least one inequality strictly satisfied.
    \end{itemize}
\end{definition}

The following definition is required to be able to distinguish two points that have the same \(\psi\) and \(h\) value.

\begin{definition}[from \cite{AuTr2018}]
    A point \(\mathbf{x}^1 \in \mathbb{R}^n\) is \textit{older} than a point \(\mathbf{x}^2 \in \mathbb{R}^n\) if it was generated before \(\mathbf{x}^2\) by an algorithm. The function \(\texttt{older}: \mathbb{R}^n \times \mathbb{R}^n \rightarrow \mathbb{R}^n\) defined as
    \[\texttt{older}(\mathbf{x}^1, \mathbf{x}^2) = \begin{cases}
        \mathbf{x}^1 & \text{if } \mathbf{x}^1 \text{ is older than } \mathbf{x}^2, \\
        \mathbf{x}^2 & \text{otherwise},
    \end{cases}\]
    returns the oldest of two points.
\end{definition}

The NM search step depends on a comparison function defined as follows.

\begin{definition}[adapted from~\cite{AuTr2018}]
Let \(\psi : \mathbb{R}^n \rightarrow \mathbb{R} \cup \{+ \infty\}\) be a scalar-valued function. The function \(\texttt{best}_\psi : \mathbb{R}^n \times \mathbb{R}^n \rightarrow \mathbb{R}^n\)
\[\texttt{best}_\psi(\mathbf{x}^1, \mathbf{x}^2) = \begin{cases}
    \mathbf{x}^1 & \text{if } \mathbf{x}^1 \prec_\psi \mathbf{x}^2 \text{ or } h(\mathbf{x}^1) < h(\mathbf{x}^2), \\
    \mathbf{x}^2 & \text{if } \mathbf{x}^2 \prec_\psi \mathbf{x}^1 \text{ or } h(\mathbf{x}^2) < h(\mathbf{x}^1), \\
    \texttt{older}(\mathbf{x}^1, \mathbf{x}^2) & \text{if } \psi(\mathbf{x}^1) = \psi(\mathbf{x}^2) \text{ or } h(\mathbf{x}^1) = h(\mathbf{x}^2),
\end{cases}\]
returns the best of two points \(\mathbf{x}^1\) and \(\mathbf{x}^2\) in \(\mathbb{R}^n\). The point \(\mathbf{x}^1\) is said to be better than \(\mathbf{x}^2\) relatively to \(\psi\) if \(\mathbf{x}^1 = \texttt{best}_\psi(\mathbf{x}^1, \mathbf{x}^2)\).
\end{definition}

Using~\cite[Proposition 4.3]{{AuTr2018}}, one can show that \(\texttt{best}_\psi\) is transitive over the set of trial points. 

Algorithm~\ref{alg:NM_search_procedure} offers a high-level description of the NM subproblem procedure for a particular function \(\psi\). It is described next in more details.

\begin{figure}[!th]
  \begin{algorithm}[H]\small
   \caption{A high-level description of the NM subproblem procedure using \(\psi\) (adapted from~\cite{AuTr2018}).}
   \begin{algorithmic}
    \STATE \textbf{Construct an initial ordered simplex \(\mathbb{X}_{s}\)}. Using the set \(\mathbb{T}_{\pi_{\text{radius}}}\) and the function \(\texttt{best}_\psi\), try to construct an ordered simplex \(\mathbb{X}_s \subseteq \mathbb{T}_{\pi_\text{radius}}\). If \(|\mathbb{X}_s| \neq n+1\), stop the search.
    \STATE \textbf{Update the simplex}. If \(\mathbb{X}_s\) is not a simplex, go to \textit{Termination}. Otherwise, reorder \(\mathbb{X}_s = \{\mathbf{x}^0, \mathbf{x}^1, \ldots, \mathbf{x}^n\}\) using \(\texttt{best}_\psi\). 
    \STATE \textbf{Determine a new candidate vertex \(\mathbf{t}\)}. Generate Nelder-Mead candidates \(\{\mathbf{x}^\text{r}_{\oplus}, \mathbf{x}^\text{e}_{\oplus}, \mathbf{x}^\text{ic}_{\oplus}, \mathbf{x}^\text{oc}_{\oplus}\} \subset M^k\). Following~\cite[Algorithm 4]{AuTr2018}, if \(\mathbb{X}_s\) is to be shrunk, go to \textit{Termination}. Otherwise, update \(\mathbf{t} \in \{\mathbf{x}^\text{r}_{\oplus}, \mathbf{x}^\text{e}_{\oplus}, \mathbf{x}^\text{ic}_{\oplus}, \mathbf{x}^\text{oc}_{\oplus}\}\). 
    \STATE \textbf{Replace the worst point}. If \(\mathbf{t} \in V^k\), go to \textit{Termination}. Otherwise, set \(\mathbf{x}^n := \mathbf{t}\) and go to \textit{Update the simplex}.
    \STATE \textbf{Termination}: Collect all solution points \(S^\text{tmp} \subset \mathcal{X} \cap M^k\) evaluated by the NM subproblem procedure. Set \(S^k := S^k \cup S^\text{tmp}\). Return.
   \end{algorithmic}
   \label{alg:NM_search_procedure}
  \end{algorithm}
  %\caption{The NM search procedure (adapted from~\cite{AuTr2018}).}
\end{figure}

At iteration \(k\), given the incumbent iterate \(\mathbf{x}^k, \Delta^k \in L^k\), and the cache \(V^k\), \(\mathbb{T}_{\pi_\text{radius}}\) is defined as:
\[\mathbb{T}_{\pi_\text{radius}} = \{\mathbf{x} \in V^k: \|\mathbf{x} - \mathbf{x}^k\|_{\infty} \leq \pi_{\text{radius}} \Delta^k\},\]
where \(\pi_{\text{radius}} \geq 1\) is a parameter controlling the size of the zone in which points can be collected. Using \(\texttt{best}_{\psi}\), the NM procedure tries to construct an ordered simplex \(\mathbb{X}_s = \{\mathbf{x}^0, \mathbf{x}^1, \ldots, \mathbf{x}^{n}\} \subseteq \mathbb{T}_{\pi_\text{radius}}\) such that \(\mathbf{x}^{j-1} = \texttt{best}_{\psi}(\mathbf{x}^{j-1}, \mathbf{x}^j)\) for \(j = 1, 2, \ldots, n\).

If \(\mathbb{X}_s\) is valid, it contains \(n+1\) elements. The NM procedure constructs the following points from \(\mathbb{X}_s\) using the standard NM algorithm:
\[\begin{cases} \mathbf{x}^\text{c} = \frac{1}{n} \sum_{j=0}^{n-1} \mathbf{x}^j & \text{the centroid of the } n \text{ best points}; \\
\mathbf{x}^\text{r} = \mathbf{x}^\text{c} + (\mathbf{x}^\text{c} - \mathbf{x}^n) & \text{the reflection candidate}; \\
\mathbf{x}^\text{e} = \mathbf{x}^\text{c} + \delta^\text{e} (\mathbf{x}^\text{c}  - \mathbf{x}^n) & \text{the expansion candidate with } \delta^\text{e} \in ]1, + \infty[; \\
\mathbf{x}^\text{oc} = \mathbf{x}^c + \delta^\text{oc} (\mathbf{x}^c - \mathbf{x}^n) & \text{the outside-contraction candidate with } \delta^\text{oc} \in ]0, 1[ \text{ and} \\
\mathbf{x}^\text{ic} = \mathbf{x}^c + \delta^\text{ic} (\mathbf{x}^\text{c} - \mathbf{x}^n) & \text{the inside-contraction candidate with } \delta^{ic} \in ]-1, 0[. \\
\end{cases}\]
The NM procedure uses the projections of \(\mathbf{x}^\text{r}\), \(\mathbf{x}^\text{e}\), \(\mathbf{x}^\text{oc}\), and \(\mathbf{x}^\text{ic}\) on the mesh \(M^k\), respectively \(\mathbf{x}^\text{r}_{\oplus}\), \(\mathbf{x}^\text{e}_{\oplus}\), \(\mathbf{x}^\text{oc}_{\oplus}\), and \(\mathbf{x}^\text{ic}_{\oplus}\) as potential candidates to evaluate.

The NM procedure performs a succession of iterative steps, where it tries to replace the worst point \(\mathbf{x}^n\) of the simplex \(\mathbb{X}_s\) by a new candidate \(\mathbf{t} \in \{\mathbf{x}^\text{r}_{\oplus}, \mathbf{x}^\text{e}_{\oplus}, \mathbf{x}^\text{ic}_{\oplus}, \mathbf{x}^\text{oc}_{\oplus}\}\). This substitution procedure is detailed in~\cite{AuTr2018}.

The NM procedure stops when:
\begin{itemize}
    \item it tries to evaluate a point already generated during the previous iterations, i.e., \(\mathbf{t} \in V^k\);
    \item \(\mathbb{X}_s\) is not a simplex anymore, due to the projection on the mesh;
    \item when the NM procedure tries to enter in the shrinking phase of the standard NM algorithm~\cite{AuTr2018} or;
    \item after the NM procedure reaches a certain number of blackbox evaluations.
\end{itemize}
The NM search collects all evaluated points as potential candidates to update the iterate list. For more details, the reader is referred to~\cite{AuTr2018}.

% ----------------------------------------------------%
\section{Computational experiments}
\label{sect:Numerical_experiments}
% ----------------------------------------------------%

This section is dedicated to the numerical experiments of new search strategies for DMulti-MADS. The first part presents the considered variants of DMulti-MADS. The second part evaluates the performance of these variants on (bound-)constrained analytical benchmarks using data profiles for multiobjective optimization. The last part shows the benefit of such search strategies on two biobjective~\cite{solar_paper} and one triobjective~\cite{AuSaZg2010a} engineering applications using convergence profiles.

This work assesses the performance of multiobjective DFO algorithms by adapting data~\cite{MoWi2009} and convergence profiles to the multiobjective context~\cite{BiLedSa2020, G-2022-10}. Both tools require a convergence test based on the hypervolume indicator~\cite{Zitzler1998}.

Given a Pareto front approximation \(Y_N \subset \mathbb{R}^m\), the hypervolume indicator of \(Y_N\) is the measure of the space in the objective space dominated by \(Y_N\) and bounded above by a objective vector \(\mathbf{u} \in \mathbb{R}^m\) such that for all \(\mathbf{y} \in Y_N\), \(\mathbf{y} \leq \mathbf{u}\). Formally,
\[HV(Y_N, \mathbf{r}) = \Lambda \left(\left\{\mathbf{v} \in \mathbb{R}^m \ | \ \exists \mathbf{y} \in Y_N : \mathbf{y} \leq \mathbf{v} \text{ and } \mathbf{v} \leq \mathbf{u}\right\}\right) = \Lambda\left(\bigcup_{\mathbf{y} \in Y_N} \left[\mathbf{y}, \mathbf{u}\right]\right),\]
where \(\Lambda(.)\) denotes the Lebesgue measure of a \(m\)-dimensional set of points. Pareto-compliant with the dominance order~\cite{Zitzler2003} and intuitive to understand, the hypervolume indicator captures several properties of a Pareto front approximation, such as cardinality, convergence, spread, and extent. All these points make the hypervolume indicator a reasonable choice for assessing the quality of a Pareto front approximation.

Given a set of problems \(\mathcal{P}\), the convergence test for a problem \(p \in \mathcal{P}\) requires a Pareto front approximation reference \(Y^p\). In a DFO context, the analytical Pareto front is not available. \(Y^p\) is then constructed by taking the union of the best feasible non-dominated points found by all considered solvers on problem \(p \in \mathcal{P}\) for a maximum budget of function evaluations, from which new dominated points are removed. If no solver manages to find a feasible solution, \(p \in \mathcal{P}\) is discarded from \(\mathcal{P}\). Once \(Y^p\) is computed, the approximated ideal objective vector
\[\tilde{\mathbf{y}}^{\text{I}, p} = \left(\min_{\mathbf{y} \in Y^p} y_1, \min_{\mathbf{y} \in Y^p} y_2,  \ldots, \min_{\mathbf{y} \in Y^p} y_{m_p}\right)^\top,\]
and the approximated nadir objective vector
\[\tilde{\mathbf{y}}^{\text{N}, p} = \left(\max_{\mathbf{y} \in Y^p} y_1, \max_{\mathbf{y} \in Y^p} y_2,  \ldots, \max_{\mathbf{y} \in Y^p} y_{m_p}\right)^\top\]
are extracted, with \(m_p\) the number of objectives associated to problem \(p \in \mathcal{P}\).

Let \(Y^e\) be the Pareto front approximation obtained by a deterministic algorithm on problem \(p \in \mathcal{P}\) after \(e\) blackbox evaluations. To avoid privileging one objective against another, one applies the transformation \(T: \mathbb{R}^m \rightarrow \mathbb{R}^m\) to each element of \(Y^e\) and \(Y^p\), and \(\tilde{\mathbf{y}}^{\text{N}, p}\), defined as
\[T(\mathbf{y}) = \begin{cases}
(\mathbf{y} - \tilde{\mathbf{y}}^{\text{I}, p}) \oslash (\tilde{\mathbf{y}}^{\text{N}, p} - \tilde{\mathbf{y}}^{\text{I},p}) & \text{if } \tilde{\mathbf{y}}^{\text{N}, p} \neq \tilde{\mathbf{y}}^{\text{I},p} \\
\mathbf{y} - \tilde{\mathbf{y}}^{\text{I}, p} & \text{otherwise},
\end{cases}\]
where \(\oslash\) is the Hadamard division operator.

Given a Pareto front approximation \(Y\), let denote
\[T(Y) = \left\{\mathbf{v} \in \mathbb{R}^m : \mathbf{v} = T(\mathbf{y}) \text{ for } \mathbf{y} \in Y \right\}.\]
A computational problem \(p \in \mathcal{P}\) is said to be solved by an algorithm with tolerance \(\varepsilon_\tau > 0\) if
\[\dfrac{HV\left(T(Y^e), T(\tilde{\mathbf{y}}^{\text{N}, p})\right)}{HV\left(T(Y^p), T(\tilde{\mathbf{y}}^{\text{N}, p})\right)} \geq 1 - \varepsilon_\tau.\]
Note that all elements of \(T(Y^e)\) dominated by \(T(\tilde{\mathbf{y}}^{\text{N}, p})\) do not contribute to the hypervolume and need to be removed before the computation. If the solver does not find at least one transformed objective vector that dominates \(T(\tilde{\mathbf{y}}^{\text{N}, p})\), \(HV\left(T(Y^e), T(\tilde{\mathbf{y}}^{\text{N}, p})\right)\) is set to \(0\).

Data profiles show the proportion of computational problems solved by an algorithm according to the number of groups of \(n+1\) function evaluations, i.e., the number of points required to build a gradient simplex in \(\mathbb{R}^n\).

\subsection{Tested DMulti-MADS variants}

The considered variants are:
\begin{itemize}
    \item DMulti-MADS \texttt{basic}: DMulti-MADS with an only speculative search, as described in~\cite{BiLedSa2020, G-2022-10}.
    \item DMulti-MADS \texttt{NM-DoM}: DMulti-MADS \texttt{basic} with the NM dominance move search strategy.
    \item DMulti-MADS \texttt{NM-Multi}: DMulti-MADS \texttt{basic} with the NM MultiMADS search strategy.
    \item DMulti-MADS \texttt{Quad-DoM}: DMulti-MADS \texttt{basic} with the quadratic dominance move search strategy.
    \item DMulti-MADS \texttt{Quad-Multi}: DMulti-MADS \texttt{basic} with the quadratic MultiMADS search strategy.
\end{itemize}
This work implements an additional quadratic search strategy adapted from~\cite{BraCu2020} called DMulti-MADS \texttt{Quad-DMS}. The procedure is given in Algorithm~\ref{alg:Quad_DMS_search_procedure}.

\begin{figure}[!th]
  \begin{algorithm}[H]\small
   \caption{A high-level description of the quadratic DMS search strategy  (adapted from~\cite{BraCu2020}).}
   \begin{algorithmic}
    \STATE \textbf{Initialization}. Given a current frame incumbent center \(\mathbf{x}^k\), its associated frame size parameter \(\Delta^k > 0\), \(\rho > 1)\) the radius factor, select a set of points \(V^{Q, k} \subseteq V^k \cap B_{\infty}(\mathbf{x}^k, \rho \Delta^k)\). Initialize the counter \(l = 0\).
    \WHILE{\(l < m\)}
        \STATE Set \(l := l + 1\). Let \(\mathcal{I}\) the set of all combinations of \(l\) quadratic models taken from the total set of \(m\) models. Set \(S := \emptyset\).
        \FOR{\(I \in \mathcal{I}\)}
            \STATE Approximately solve
            \[\mathbf{s}_I \approx \arg \min_{\mathbf{x} \in \mathbb{R}^n} Q_{\psi_{I}}(\mathbf{x})
                \text{ s.t. } \begin{cases} Q_{c_j}(\mathbf{x}) \leq 0, \ j \in \mathcal{J}, \\
                \mathbf{x} \in \mathcal{X}, \\
                \mathbf{x} \in \mathcal{B}_{\infty}(\mathbf{x}^{k, \text{c}}, \rho \Delta^k).
                \end{cases}\]
            where \(Q_{\psi_I}\) is the quadratic model associated to \(\psi_I : \mathbf{x} \mapsto \displaystyle\max_{i \in I} f_i(\mathbf{x})\), built from \(V^{Q,k}\).
            Project \(\mathbf{s}_I\) onto \(M^k\) and update \(S := S \cup \{\mathbf{s}_I\}\).
        \ENDFOR
        \STATE \textbf{Check for success}. Evaluate \(f\) and \(h\) at each point of \(S\). Update \(V^{Q, k} := V^{Q, k} \cup S\) and \(S^k := S^k \cup S\). Compute \(L^{\text{trial}}\) by removing all dominated points from \(L^k \cup \left\{(\mathbf{s}, \Delta) : \mathbf{s} \in S \text{ and } \Delta \in \left\{\Delta^k, \tau^{-1} \Delta^k \right\}\right\}\). If \(L^{\text{trial}} \neq L^k\), return.
    \ENDWHILE
   \end{algorithmic}
   \label{alg:Quad_DMS_search_procedure}
  \end{algorithm}
  %\caption{The Quadratic DMS search strategy (adapted from~\cite{BraCu2020}).}
\end{figure}

At iteration \(k\), the quadratic DMS search procedure incrementally solves successive quadratically constrained quadratic subproblems, starting from the minimization of each individual objective quadratic model (i.e., the level \(l = 1\)) until the minimization of all objective quadratic models simultaneously (i.e., the level \(l = m\)). When all objective combinations for a given level \(l\) are exhausted, the quadratic DMS search procedure checks whether a new non-dominated solution has been generated. If so, the procedure stops. Otherwise, the procedure moves to the next level. In the worst case, this search iterates over all levels, resulting in the optimization of \(2^m - 1\) subproblems. This issue remains limited for a small number of objectives, i.e., \(m \leq 4\).

Note that contrary to~\cite{BraCu2020}, a success for the quadratic DMS search is not equivalent to a success for the complete DMulti-MADS iteration. In particular, DMulti-MADS may continue with another search step or poll step even if the quadratic DMS search is successful.

All of these search strategies are implemented in the \nomad software~\cite{nomad4paper}, version \texttt{4.5.0}, which can be found at~\url{https://github.com/bbopt/nomad}. For the rest of this work, all variants use an opportunistic strategy and an OrthoMads poll step with \(n+1\) directions without quadratic models~\cite{AuIaLeDTr2014}. The search strategies use the same default parameters as their single-objective counterparts implemented in \nomad. For example, all quadratically constrained quadratic subproblems are solved with MADS.

An implementation of the benchmark functions used in the next subsection can be found
at~\url{https://github.com/bbopt/DMultiMads_search_benchmarks}.

\subsection{Computational experiments on synthetic benchmarks}

In this subsection, this work considers two analytical benchmark sets:
\begin{itemize}
    \item a set of \(100\) bound-constrained multiobjective optimization problems taken from~\cite{CuMaVaVi2010}, with \(n \in [1, 30]\) and \(m \in \{2,3,4\}\);
    \item a set of \(214\) constrained multiobjective optimization problems proposed in~\cite{LiLuRi2016}, based on the previous benchmark suite, with \(n \in [3, 30]\), \(m \in \{2,3,4\}\) and \(|\mathcal{J}| \in [3, 30]\).
\end{itemize}
For each problem \(p \in \mathcal{P}\), each variant is run once with a maximal budget of $30,000$ blackbox evaluations, starting from the same \(n_p\) points by dividing the line connecting the lower bound and upper bound
into \(n_p\) equal spaces~\cite{CuMaVaVi2010}, with \(n_p\) the number of variables of problem \(p \in \mathcal{P}\).

\begin{figure}[!ht]
		\centering
		\subfigure[\(\varepsilon_{\tau}=10^{-2}\)]{
			\includegraphics[scale=0.7]{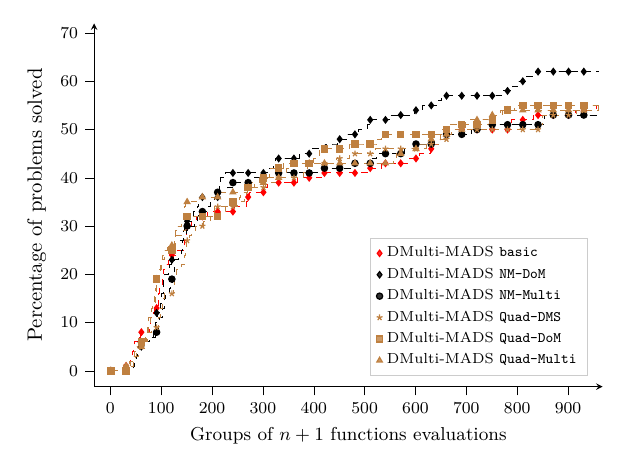}}
		\quad
		\subfigure[\(\varepsilon_{\tau}=5 \times 10^{-2}\)]{
			\includegraphics[scale=0.7]{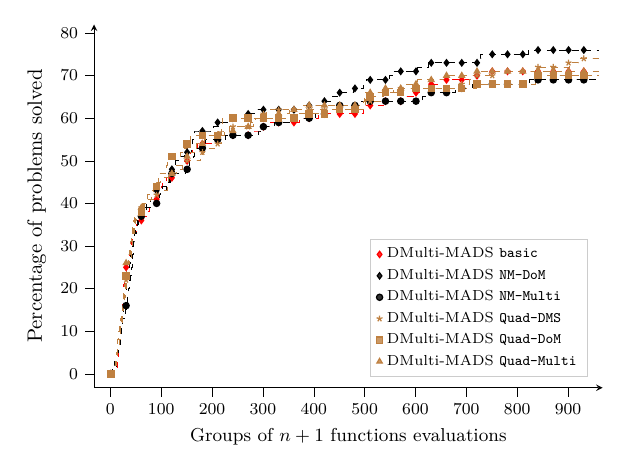}}
		\quad
		\subfigure[\(\varepsilon_{\tau}=10^{-1}\)]{
			\includegraphics[scale=0.7]{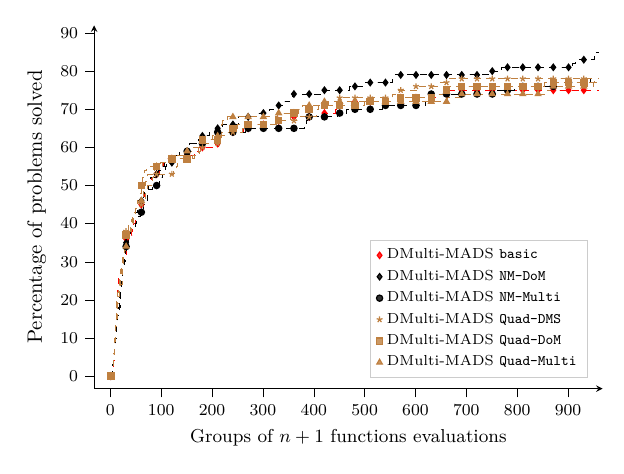}}
		\caption{Data profiles obtained from 100 multiobjective bound-constrained analytical problems taken from~\cite{CuMaVaVi2010} for DMulti-MADS \texttt{basic}, \texttt{NM-DoM}, \texttt{NM-Multi}, \texttt{Quad-DMS}, \texttt{Quad-DoM}, and \texttt{Quad-Multi} variants with tolerance \(\varepsilon_\tau \in \{10^{-2}, 5 \times 10^{-2}, 10^{-1}\}\).}
		\label{fig:dmultimads_variants_synthetic_benchmarks_unconstrained}
\end{figure}

The data profiles in Figure~\ref{fig:dmultimads_variants_synthetic_benchmarks_unconstrained} show that the addition of such search strategies has a positive impact on the performance of DMulti-MADS for all tolerances \(\varepsilon_\tau \in \{10^{-2}, 5 \times 10^{-2}, 10^{-1}\}\). The NM dominance move search strategy is the most efficient on the bound-constrained benchmark suite for all tolerances for a medium to large budget of blackbox evaluations. It solves up to \(8-10\%\) more problems for the lowest tolerance \(\varepsilon_\tau = 10^{-2}\) and about \(5\%\) more problems for \(\varepsilon_\tau \in \{5 \times 10^{-2}, 10^{-1}\}\) than DMulti-MADS \texttt{basic}. It is followed by DMulti-MADS \texttt{Quad-DoM} and \texttt{Quad-DMS}, which perform better than DMulti-MADS \texttt{basic} for medium to high tolerance \(\varepsilon_\tau\) considered.

\begin{figure}[!ht]
		\centering
		\subfigure[\(\varepsilon_{\tau}=10^{-2}\)]{
			\includegraphics[scale=0.7]{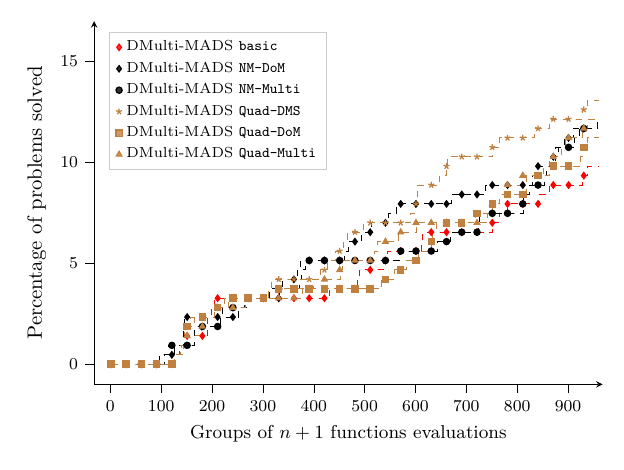}}
		\quad
		\subfigure[\(\varepsilon_{\tau}=5 \times 10^{-2}\)]{
			\includegraphics[scale=0.7]{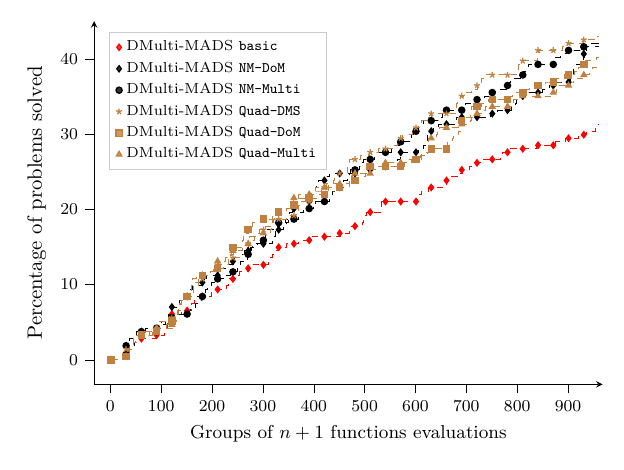}}
		\quad
		\subfigure[\(\varepsilon_{\tau}=10^{-1}\)]{
			\includegraphics[scale=0.7]{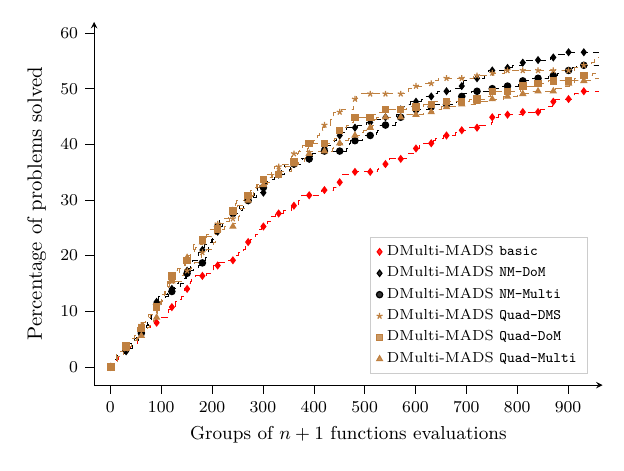}}
		\caption{Data profiles obtained from 214 multiobjective analytical problems taken from~\cite{LiLuRi2016} for DMulti-MADS \texttt{basic}, \texttt{NM-DoM}, \texttt{NM-Multi}, \texttt{Quad-DMS}, \texttt{Quad-DoM}, and \texttt{Quad-Multi} variants with tolerance \(\varepsilon_\tau \in \{10^{-2}, 5 \times 10^{-2}, 10^{-1}\}\).}
		\label{fig:dmultimads_variants_synthetic_benchmarks_constrained}
\end{figure}

This performance gain is more noticeable when solving constrained analytical problems. As shown in Figure~\ref{fig:dmultimads_variants_synthetic_benchmarks_constrained}, all variants with quadratic or NM search strategies enabled perform better than DMulti-MADS \texttt{basic} for all evaluations budgets for medium (\(\varepsilon_\tau = 5 \times 10^{-2}\)) to high tolerance (\(\varepsilon_\tau = 10^{-1}\)), up to \(5\%\) more problems solved for the less efficient search strategy. Note, however, that DMulti-MADS has some difficulty with this benchmark set, solving less than \(60\%\) of the problems for the lowest tolerance \(\varepsilon_\tau = 10^{-1}\). DMulti-MADS \texttt{Quad-DMS} is slightly more efficient than the other variants for a medium to large budget of evaluations in this case for all considered tolerances.

\subsection{Computational experiments on practical engineering problems}

To validate the performance of the new search strategies, this work considers three ``real-world'' engineering applications: two biobjective problems, {\sf solar 8.1} and {\sf solar 9.1}~\cite{solar_paper}, and one triobjective application STYRENE~\cite{AuSaZg2010a}. To follow the progression of the different DMulti-MADS variants, this work uses convergence profiles for multiobjective optimization. Because the evaluation of such blackboxes is more expensive than for previous benchmark sets, computing data profiles cannot be performed in a reasonable time for this research.

Convergence profiles track the progression of a multiobjective optimization algorithm along the iterations by computing the following normalized hypervolume value
\[\dfrac{HV\left(T(Y^e), T(\tilde{\mathbf{y}}^{\text{N}, p})\right)}{HV\left(T(Y^p), T(\tilde{\mathbf{y}}^{\text{N}, p})\right)},\]
as a function of the number of blackbox evaluations \(e\), where \(Y^p\) is the Pareto front reference for problem \(p\), and \(Y^e\) is the Pareto front approximation obtained at evaluation \(e\) by the algorithm. The normalized hypervolume indicator is monotonically increasing and bounded above by 1. As mentioned previously at the beginning of Section~\ref{sect:Numerical_experiments}, a normalized hypervolume value equal to 0 indicates that a method has not generated a feasible objective vector that is dominated by the approximated nadir objective vector \(\tilde{\mathbf{y}}^{\text{N}, p}\) of problem \(p\).

{\sf solar 8.1} and {\sf solar 9.1} are two deterministic biobjective problems derived from the design of a solar plant system modeled by a numerical simulation~\cite{solar_paper}. The simulator consists of three subsystems. The heliostats field collects solar rays that are sent to a receiver. The second subsystem converts the heat generated by the receiver into thermal energy. The power-block uses this energy to produce water steam, which rotates turbines to generate electricity. Interested readers are referred to~\cite{solar_paper} for more details. {\sf solar~8.1} aims to maximize heliostats field performance and minimize cost. It possesses 13 variables, 2 of which are integer. {\sf solar 9.1} aims to maximize power production and minimize cost. It possesses 29 variables, 7 of which are integer.

DMulti-MADS does not handle integer variables.  Then, for both problems, integer values are kept constant to their initial value (they can be found at~\cite{G-2022-10}), resulting in blackboxes with 11 and 22 continuous input variables. All variants are run with a budget of $5,000$ evaluations and start from the same infeasible initial point.

\begin{figure}[!th]
\centering
\subfigure[]{\includegraphics[scale=0.7]{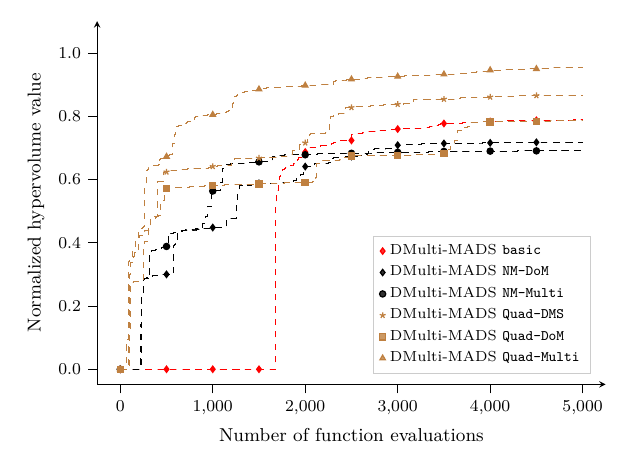}
             \label{fig:solar8_convergence_profiles}}
\quad
\subfigure[]{\includegraphics[scale=0.7]{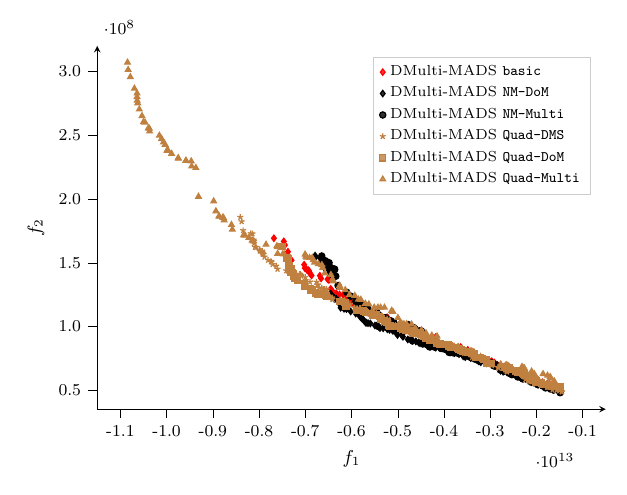}
             \label{fig:solar8_pareto_fronts}}
\caption{(a) On the left, convergence profiles for the {\sf solar 8.1} problem using DMulti-MADS \texttt{basic}, \texttt{NM-DoM}, \texttt{NM-Multi}, \texttt{Quad-DMS}, \texttt{Quad-DoM}, and \texttt{Quad-Multi} variants for a maximal budget of $5,000$ evaluations. (b) On the right, Pareto front approximations obtained at the end of the resolution of {\sf solar 8.1} for DMulti-MADS \texttt{basic}, \texttt{NM-DoM}, \texttt{NM-Multi}, \texttt{Quad-DMS}, \texttt{Quad-DoM}, and \texttt{Quad-Multi} in the objective space.}
\end{figure}

From Figure~\ref{fig:solar8_convergence_profiles}, DMulti-MADS with the quadratic MultiMADS search strategy performs better than the other variants, followed by the quadratic DMS search strategy. Turning to Figure~\ref{fig:solar8_pareto_fronts}, the quadratic MultiMADS search manages to extend the Pareto front approximation towards \(f_1\), by exploiting the nature of the current frame incumbent. The quadratic DMS search strategy is slower to extend the Pareto front approximation in this case, because it is myopic in its exploration of the objective space.

\begin{figure}[!th]
	\centering
	\subfigure[]{\includegraphics[scale=0.7]{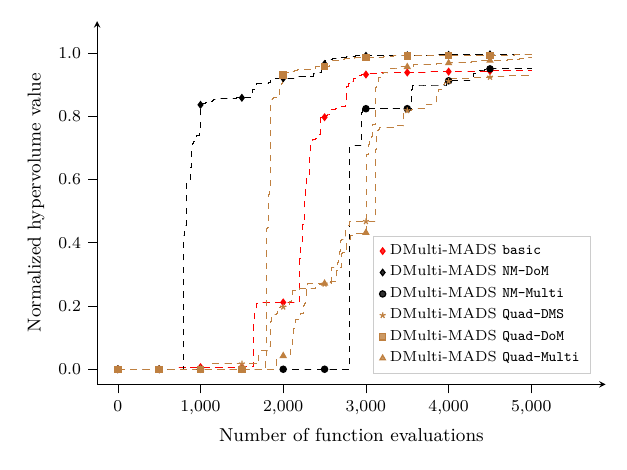}
		\label{fig:solar9_convergence_profiles}}
	\quad
	\subfigure[]{\includegraphics[scale=0.7]{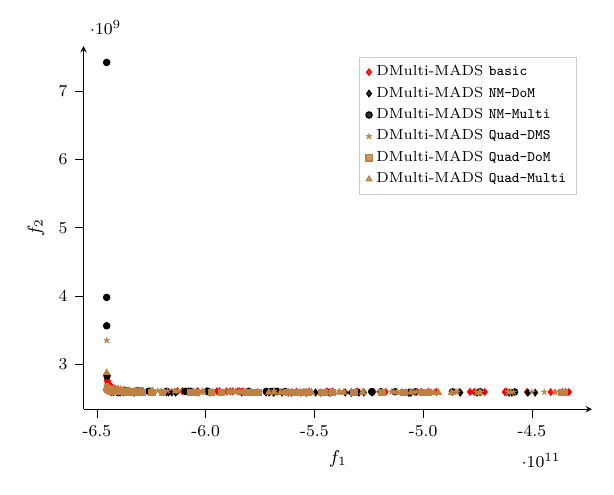}
		\label{fig:solar9_pareto_fronts}}
    	\caption{(a) On the left, convergence profiles for the {\sf solar 9.1} problem using DMulti-MADS \texttt{basic}, \texttt{NM-DoM}, \texttt{NM-Multi}, \texttt{Quad-DMS}, \texttt{Quad-DoM}, and \texttt{Quad-Multi} variants for a maximal budget of $5,000$ evaluations. (b) On the right, Pareto front approximations obtained at the end of the resolution of {\sf solar 9.1} for DMulti-MADS \texttt{basic}, \texttt{NM-DoM}, \texttt{NM-Multi}, \texttt{Quad-DMS}, \texttt{Quad-DoM}, and \texttt{Quad-Multi} in the objective space.}
    \label{fig:solar9_convergence_profiles_and_pareto_fronts}
\end{figure}

Results for {\sf solar 9.1} are presented in Figure~\ref{fig:solar9_convergence_profiles_and_pareto_fronts}. The two dominance move search strategies are the most efficient, as illustrated in Figure~\ref{fig:solar9_convergence_profiles}. The Pareto front approximation in Figure~\ref{fig:solar9_pareto_fronts} is ``flat'', which could favor such strategies. \\

The last problem considered in this subsection is STYRENE~\cite{AuSaZg2010a}. It is based on a numerical simulator that models styrene production, implemented in \texttt{C++}, and available at \url{github.com/bbopt/styrene}. 

The chemical production of styrene involves four steps: reactants preparation, catalytic reactions, and two distillation phases. The first distillation allows the recovery of styrene, while the second is aimed at the recovery of benzene. During this second distillation, one can recycle unreacted ethylbenzaline, which can be reinjected into the styrene production as an initial reactant. The goal is to maximize three objectives simultaneously: the net present value associated with the process (\(f_1\)), the purity of the styrene produced (\(f_2\)) and the overall conversion of ethylbenzene to styrene (\(f_3\)). This problem has eight continuous variables and nine inequality constraints, related to industrial regulations and economic context. Interested readers are referred to~\cite{AuSaZg2010a} for more details. The simulation may crash, i.e., it may not return finite numerical output values, due to the presence of hidden constraints in the blackbox.

All DMulti-MADS variants are affected a maximal budget of $20,000$ blackbox evaluations and start from the same initial feasible point.

\begin{figure}[!th]
	\centering
	\subfigure[]{\includegraphics[scale=0.7]{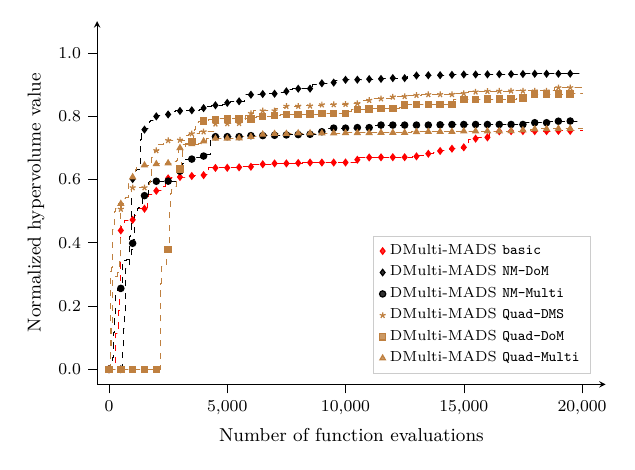}
		\label{fig:styrene_convergence_profiles}}
	\quad
	\subfigure[]{\includegraphics[scale=0.59]{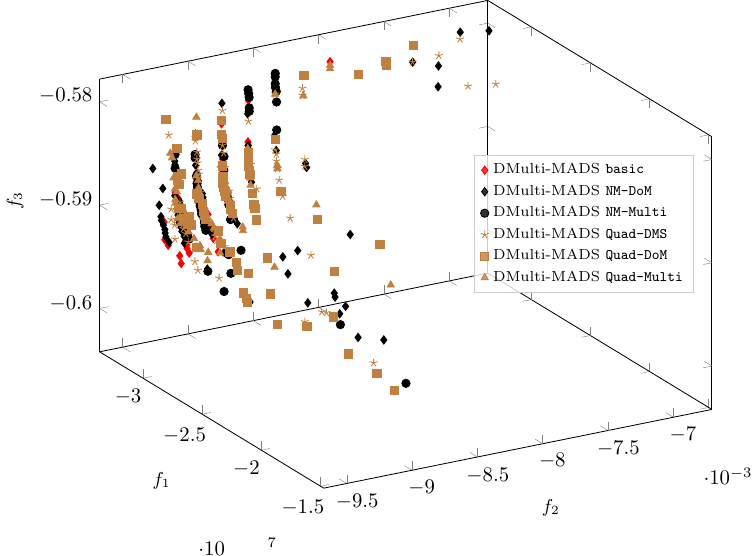}
		\label{fig:styrene_pareto_fronts}}
    \caption{(a) On the left, convergence profiles for the STYRENE problem using DMulti-MADS \texttt{basic}, \texttt{NM-DoM}, \texttt{NM-Multi}, \texttt{Quad-DMS}, \texttt{Quad-DoM}, and \texttt{Quad-Multi} variants for a maximal budget of $5,000$ evaluations. (b) On the right, Pareto front approximations obtained at the end of the resolution of STYRENE for DMulti-MADS \texttt{basic}, \texttt{NM-DoM}, \texttt{NM-Multi}, \texttt{Quad-DMS}, \texttt{Quad-DoM}, and \texttt{Quad-Multi} in the objective space.}
\end{figure}

From Figure~\ref{fig:styrene_convergence_profiles}, one can see that all variants of DMulti-MADS with search strategies outperform DMulti-MADS \texttt{basic}. The dominance move search strategies are the most efficient. The Pareto front approximations obtained on Figure~\ref{fig:styrene_pareto_fronts} show that all search strategies introduced in this work are able to capture a larger part of the Pareto front reference than DMulti-MADS standard algorithm.

% ----------------------------------------------------%
\section{Closing Remarks}
% ----------------------------------------------------%

This work proposes two search strategies for the DMulti-MADS constrained algorithm for multiobjective optimization, that generalizes previous heuristics proposed for the MADS algorithm~\cite{AuTr2018,CoLed2011,CuRoVi10}. The first uses quadratic models, built from previous evaluated points, that act as surrogates of the objective and constraint functions, providing new promising candidates. The other relies on a Nelder-Mead-based sampling strategy. Both strategies are built around the resolution of single-objective subproblems derived from the original multiobjective problem, in the lineage of~\cite{AuSaZg2008a, AuSaZg2010a, BraCu2020}. This work also introduces a single-objective formulation framework that generalizes~\cite{AuSaZg2008a, AuSaZg2010a}.

Experiments on (bound-)constrained analytical benchmarks and three engineering problems show that the search strategies bring significant improvement over the baseline version of the DMulti-MADS algorithm. They also highlight the fact that none of these strategies is superior for all problems.

\nomad chooses as default search strategies \texttt{NM-DoM} and \texttt{Quad-Multi} because they allow a gain without sacrificing computational time (e.g., in solving quadratic subproblems). A user may modify these default options to obtain better results according to his problem.

We emphasize that these search strategies are not limited to DMulti-MADS and can also be implemented for other direct search algorithms for multiobjective optimization, such as DMS~\cite{CuMaVaVi2010}. Future work includes the exploitation of parallelism, in the continuation of~\cite{Tavares2023}, and the integration of the single-objective formulations in the construction of smarter poll strategies based on quadratic models~\cite{AuIaLeDTr2014}.

\subsection*{Funding}

This work is supported by the NSERC/Mitacs Alliance grant (\#571311--21) in collaboration with Hydro-Qu\'ebec and a postdoctoral fellowship of GERAD (FRQNT--309029).

\subsection*{Conflicts of interest}
The authors report no conflict of interest.

\subsection*{Data availability statement}

The \nomad solver is available at~\url{https://github.com/bbopt/nomad}.
The {\sf solar 8.1} and {\sf solar 9.1} problems are available at~\href{https://github.com/bbopt/solar}{\bl \url{github.com/bbopt/solar}}.
The STYRENE problem is available at~\url{github.com/bbopt/styrene}.
The implementation of the benchmark functions can be found
at~\url{https://github.com/bbopt/DMultiMads_search_benchmarks}.

% biblio
% \newpage
\bibliographystyle{plain}
\bibliography{bibliography}

\newpage

\end{document}